\documentclass[11pt,a4paper,leqno,english]{article}

\usepackage{amsmath,amsthm}
\usepackage{amssymb,esint}
\usepackage{amscd}
\usepackage{mathrsfs}
\usepackage{enumerate}
\usepackage{anysize}
\marginsize{3.5cm}{3.5cm}{3cm}{3cm}

\usepackage[]{hyperref}
\hypersetup{
    colorlinks=true,       
    linkcolor=red,          
    citecolor=blue,        
    urlcolor=cyan           
}

\usepackage[titles]{tocloft}

\theoremstyle{plain}
\newtheorem{theo}{Theorem}[section]
\newtheorem{prop}[theo]{Proposition}
\newtheorem{lemm}[theo]{Lemma}
\newtheorem{coro}[theo]{Corollary}

\newtheorem{defi}[theo]{Definition}
\theoremstyle{definition}
\newtheorem{rema}[theo]{Remark}
\newtheorem{nota}[theo]{Notation}

\DeclareMathOperator{\cnx}{div}
\DeclareMathOperator{\cn}{div}

\DeclareMathOperator{\diff}{d}

\DeclareSymbolFont{pletters}{OT1}{cmr}{m}{sl}
\DeclareMathSymbol{s}{\mathalpha}{pletters}{`s}

\def\BMO{\rm{BMO}}
\def\ME{\widetilde{E}_k}

\def\dt{\diff \! t}
\def\dx{\diff \! x}
\def\dz{\diff \! z}

\def\dy{\diff \! y}

\def\dydx{\diff \! y \diff \! x}

\def\Boundary{\mathcal{B}}
\def\B{B }

\def\defn{\mathrel{:=}}

\def\eps{\varepsilon}

\def\bla{\big\lvert}

\def\bra{\big\lvert}

\def\la{\left\lvert}
\def\lA{\left\lVert}
\def\ra{\right\lvert}
\def\rA{\right\lVert}
\def\le{\leq}
\def\les{\lesssim}
\def\ma{a}
\def\mez{\frac{1}{2}}
\def\ra{\right\rvert}
\def\rA{\right\rVert}

\def\xN{\mathbf{N}}
\def\xR{\mathbf{R}}

\def\xT{\mathbf{T}}
\def\xZ{\mathbf{Z}}

\def\ba{\begin{align}}

\def\bad{\begin{aligned}}
\def\be{\begin{equation}}

\def\defn{\mathrel{:=}}

\def\dsigma{\diff \! \sigma}
\def\dt{\diff \! t}
\def\dx{\diff \! x}

\def\dy{\diff \! y}

\def\dydx{\diff \! y \diff \! x}

\def\e{\eqref}
\def\ea{\end{align}}
\def\ead{\end{aligned}}

\def\ee{\end{equation}}
\def\eps{\varepsilon}

\def\fract{\frac{\diff}{\dt}}
\def\fractt{\frac{\diff^2}{\dt^2}}

\def\la{\left\vert}
\def\lA{\left\Vert}
\def\le{\leq}
\def\les{\lesssim}
\def\mez{\frac{1}{2}}

\def\ra{\right\vert}
\def\rA{\right\Vert}

\def\xN{\mathbb{N}}
\def\xR{\mathbb{R}}

\def\xT{\mathbb{T}}
\def\xZ{\mathbb{Z}}

\numberwithin{equation}{section}

\title{Virial theorems and equipartition of energy for water-waves}

\author{ Thomas Alazard and Claude Zuily}
\date{\empty}

\begin{document}

\clearpage

\maketitle

\begin{abstract}
We study several different aspects of the energy equipartition principle for water waves. We prove a virial identity that 
implies that the potential energy is equal, on average, to a modified version of the kinetic energy. 
This is an exact identity for the complete nonlinear water wave problem, which is valid for arbitrary solutions. 
As an application, we obtain non-perturbative results justifying the formation of bubbles for the 
free-surface Rayleigh-Taylor instability, for any non-zero initial data. We also derive exact virial identities involving higher order energies. The fact that such exact identities are valid for nonlinear equations is new and general: as explained in a companion paper, similar identities can be derived for many other nonlinear equations. 
We illustrate this result by an explicit computation for standing waves. As 
side results, 
we prove trace inequalities for harmonic functions in Lipschitz domains which are optimal with respect to the dependence in the Lipschitz norm of the graph.\end{abstract}

\setcounter{tocdepth}{1}
\tableofcontents

\clearpage

\section{Introduction}

This article is motivated by the study of the principle of equipartition of energy for water waves. In a broad sense, this principle states that energy must be equally distributed between potential and kinetic energy (this principle has its origin in a question which was much debated in statistical mechanics at the beginning of the twentieth century, see~\cite{Rayleigh1900VI,Tolman1918,Pais1979}). 
For many linear wave-type equations, equipartition 
holds in various senses as shown, for instance by 
Brodsky \cite{Brodsky1967}, 
Lax and Phillips \cite{LaxPhillips1967}, 
Goldstein \cite{Goldstein1969}, Duffin \cite{Duffin1970} or Vega 
and Visciglia \cite{VegaVisciglia2008JFA}. Equipartition of energy has 
also been considered for solutions of hyperbolic systems or other dispersives equations, see  Bachelot \cite{Bachelot1987}, 
Dassios and Grillakis \cite{DassiosGrillakis1984}, Glassey and Strauss \cite{Glassey-Strauss-1979}, Strichartz~\cite{Strichartz-equi} or Delort~\cite{Delort-equi}. 
In particular, the principle of equipartition of energy is valid for many equations describing 
water waves in certain asymptotic regimes. 

We intend to study the full nonlinear problem, i.e.\ the incompressible Euler equation with free surface. 
In this case, the question of equipartition of energy 
was first studied by Rayleigh in 1911~\cite{Rayleigh1911xxiii}. He showed that equipartition does not hold in general. 
More precisely, he showed that it is not satisfied by the famous approximate solutions found by Stokes~\cite{Stokes}.  
After this work, this question has been widely discussed by several mathematicians and meteorologists 
such as Platzman \cite{Platzman1947} and Starr \cite{Starr1947}. 
Rayleigh result was later confirmed by Mack and Jay~\cite{MackJay} for approximate standing water waves. 
It follows from these works 
that the difference between the mean kinetic energy and the mean potential energy 
is proportional to the fourth power of the wave amplitude.

The first main result of this paper is that, despite these negative results, there is 
an unexpected exact identity that states that, on average, the potential energy is 
equal to a modified kinetic energy. Our proof is based on Zakharov's observation that the water wave system is 
Hamiltonian in nature. In particular, we will obtain the principal identity alluded to 
above in the form of a virial theorem (see~Appendix~\ref{S:Virial} for comments about virial theorems). 
We will also establish several virial identities which show that 
there are also equipartition of energy principles which govern higher order energies, from which we will 
deduce several surprising coercive estimates. 

It has long been observed by Levine~\cite{Levine,Levine2}, Glassey~\cite{Glassey-blowup} and Sideris~\cite{SiderisCMP} that such virial identities can be applied to study blow-up phenomena. In this direction, we will see that the virial identity can be 
applied to study another problem initiated by Rayleigh~\cite{Rayleigh1879}, justifying in a quantitative way 
the apparition of bubbles or spikes in the Rayleigh-Taylor instability. 

Eventually, we will prove several new inequalities for the trace of harmonic functions. 

\subsection{The equations}

Consider an 
incompressible 
liquid occupying a time dependent fluid domain $\Omega$, 
located underneath a free surface $\Sigma$ given as a graph  and above a fixed flat bottom 
$\Gamma$ such that, at a given time $t\ge 0$,
\begin{align*}
\Omega(t)&= \{\, (x,y)\in \xT^d \times \xR\, : \, -h<y<\eta(t,x)\,\},\\
\Sigma(t)&= \{\, (x,y)\in \xT^d \times \xR\, : \, y=\eta(t,x)\,\},\\
\Gamma&=\{\, (x,y)\in \xT^d \times \xR\, : \, y=-h\,\},
\end{align*}
where $\xT^d$ denotes a $d$-dimensional torus with $d\ge 1$, $x$ (resp.\ $y$) is the horizontal (resp.\ vertical) space variable, 
$h$ is the depth (we allow $h=+\infty$ 
in the infinite depth case), 
the free surface elevation $\eta$ is an unknown function. For our problem, spatial scales may be chosen so that $\xT^d=\xR^d/(2\pi\xZ)^d$ without loss of generality. 

The fluid we will study will be assumed to be 
subjected to the force of gravity. 
In the Eulerian coordinate system, in addition to the the free surface elevation~$\eta$, the 
unknowns are the velocity field 
$u=(u_1,\ldots,u_{d+1})\colon \Omega\rightarrow \xR^{d+1}$ and the scalar pressure $P\colon\Omega\rightarrow \xR$. We assume that they satisfy the 
incompressible Euler equations of fluid mechanics, with the solid wall boundary condition on the bottom:
\begin{equation}\label{Euler}
\left\{
\begin{aligned}
&\partial_t u+(u\cdot\nabla_{x,y}) u+\nabla_{x,y}(P+gy)=0 \quad&&\text{in }\Omega,\\
&\cnx_{x,y}u=0 \quad&&\text{in }\Omega,\\
&u_{d+1}=0 \quad&&\text{on }\Gamma,
\end{aligned}
\right.
\end{equation}
where $\nabla_{x,y}=(\nabla_x,\partial_y)$, $g$ is the acceleration of gravity,   $P\colon\Omega\to\xR$ is the pressure  and 
$u\cdot\nabla_{x,y}=u_1\partial_{x_1}+\cdots+u_{d+1}\partial_y$. Moreover, we assume that the flow is irrotational. Then  $u=\nabla_{x,y}\phi$ for some 
potential function 
$\phi=\phi(t,x,y)$ satisfying
\begin{equation}\label{t5}
\left\{
\begin{aligned}
&\partial_{t} \phi +\mez \la \nabla_{x,y}\phi\ra^2 +P +g y = 0 \quad&&\text{in }\Omega,\\
&\Delta_{x,y}\phi=0\quad&&\text{in }\Omega,\\ 
&\partial_y \phi =0 \quad&&\text{on }\Gamma,
\end{aligned}
\right.
\end{equation}
where $\Delta_{x,y}=\Delta_x^2+\partial_y^2$. 
The first equation is the usual Bernoulli's equation 
(we assumed that the right-hand side of this equation is $0$ but, since 
the fluid domain is bounded, we could allow any function depending only on time).

The water-wave equations are then given by two boundary conditions 
on the free surface. Firstly, we assume that the normal 
velocity of the free surface is equal to the normal 
component of the fluid velocity $u$ on the free surface. 
Denote by $n$ the outward unit normal to $\Sigma$, given by
$$
n=\frac{1}{\sqrt{1+|\nabla_x\eta|^2}} \begin{pmatrix} -\nabla_x \eta \\ 1 \end{pmatrix}.
$$
It follows that
\begin{equation}\label{kinematics}
\frac{\partial \eta}{\partial t} =\sqrt{1+|\nabla_x \eta|^2}\,(U\cdot n) \qquad\text{with}\quad 
U=u\arrowvert_{y=\eta}.
\end{equation}
Hereafter, given a function $f=f(x,y)$ and a function $\theta=\theta(x)$, 
we use $f\arrowvert_{y=\theta}$ as a short notation for the function $x\mapsto f(x,\theta(x))$. 
In terms of the velocity potential, \e{kinematics} simplifies to
\be\label{t8}
\begin{aligned}
\partial_{t} \eta &= \sqrt{1+|\nabla_x\eta|^2}\, \partial_n\phi\arrowvert_{y=\eta},\\
&= \partial_y\phi (t,x,\eta(t,x))-
\nabla_x\eta(t,x)\cdot(\nabla_x\phi)(t,x,\eta(t,x)),
\end{aligned}
\ee
where $\partial_n$ is the normal derivative: $\partial_n=n\cdot\nabla_{x,y}$. 

The final equation expresses the balance of forces across the free surface. In the case of gravity water waves this reads:
\be\label{t9}
P\arrowvert_{y=\eta}=0.
\ee

\subsection{Hamiltonian formulation}
Consider a solution of the free surface 
Euler equations. 
At a given time~$t$, we define its kinetic energy $E_k(t)$ and its potential energy $E_p(t)$ by
\begin{align}
E_k(t)&=\mez \iint_{\Omega(t)}\la u(t,x,y)\ra^2\dydx =\mez\iint_{\Omega(t)}\la \nabla_{x,y}\phi(t,x,y)\ra^2\dydx,\label{E1}\\
E_p(t)&=\frac{g}{2}\int_{\xT^d}\eta(t,x)^2\dx,\label{E2}
\end{align}
where $\la\cdot\ra$ denotes the Euclidean norm on $\xR^{d+1}$. 
The total energy is then, by definition,
\be\label{E3}
E(t)=E_k(t)+E_p(t).
\ee
It is well known that the total energy is conserved (see \e{conservedEproof}) that is,
\be\label{conservedE}
\fract E(t)=0.
\ee

Zakharov~(\cite{Zakharov1968,Zakharov}) observed that the later identity is in fact associated to an Hamiltonian formulation. To 
explain this, 
following Craig--Sulem~\cite{CrSu}, 
we begin by introducing the Dirichlet-to-Neumann operator to reduce the water-wave equations to a system on the free surface. 
The idea is that, 
since the velocity potential $\phi$ 
is harmonic and satisfies a Neumann boundary condition on $\Gamma$, 
it is fully determined by its evaluation at the free surface. 
We thus work below with
$$
\psi(t,x)\defn\phi(t,x,\eta(t,x)),
$$
and introduce the Dirichlet-to-Neumann operator, denoted by~$G(\eta)$, relating 
$\psi$ to the normal derivative of the potential by 
(see Appendix~\S\ref{Appendix:DN} for details),
$$
G(\eta)\psi =\sqrt{1+|\nabla_x \eta|^2}\partial_n\phi\arrowvert_{y=\eta}.
$$
Then, as explained in Appendix~\S\ref{Appendix:DN}, it follows from \e{t8} and \e{t5}--\e{t9} that
\be\label{systemT}
\left\{
\begin{aligned}
&\partial_t \eta=G(\eta)\psi,\\
&\partial_t \psi+g\eta  +N(\eta,\psi)=0,
\end{aligned}
\right.
\ee
where
\begin{align}
N(\eta,\psi)=\mathcal{N} \big\arrowvert_{y=\eta}\quad\text{with}\quad
\mathcal{N}=
\mez\la\nabla_x\phi\ra^2-\mez (\partial_y\phi)^2+(\partial_y\phi)(\nabla_x\eta \cdot \nabla_x \phi)
.\label{t90}
\end{align}
Introduce the  the energy
$$
E(t) = \mathcal{E}(\eta(t,\cdot), \psi(t, \cdot)),
$$
where
$$
\mathcal{E}(\eta , \psi) =\frac{g}{2}\int_{\xT^d} \eta^2(x)\dx +\mez\int_{\xT^d} (\psi G(\eta)\psi)(x)\dx.
$$
Then~(\cite{Zakharov1968,Zakharov}) 
the system~\e{systemT} has the following hamiltonian structure:
$$
\partial_t\eta=\frac{\delta \mathcal{E}}{\delta \psi}(\eta(t,\cdot), \psi(t, \cdot))\quad ;\quad 
\partial_t\psi=-\frac{\delta \mathcal{E}}{\delta \eta}(\eta(t,\cdot), \psi(t, \cdot)).
$$
Notice that this immediately implies \e{conservedE}.

\begin{defi}[regular solutions]\label{defi:regular}
$i)$ For $s\ge 0$  we denote by $H^s(\xT^d)$ the Sobolev space of those functions $u\in L^2(\xT^d)$ such that
$$
\lA u\rA_{H^s}^2\defn \sum_{\xi\in \xZ^d}(1+|\xi|^2)^s\big\lvert \widehat{u}(\xi)\big\rvert^2<+\infty\quad\text{with}\quad \widehat{u}(\xi)=\frac{1}{(2\pi)^d}\int_{\xT^d}e^{-ix\cdot\xi}u(x)\dx.
$$

$ii)$ Let $h\in (0,+\infty]$. 
We will say that $(\eta,\psi)$ is a regular solution of the water-wave system \eqref{systemT} provided 
that:
\begin{enumerate}
\item 
$(\eta,\psi)\in C^0([0,T], H^{s}(\xT^d)\times H^s(\xT^d))$ for some $T>0$ and some 
$s>2+d/2$,
\item for all $t\in [0,T]$, there holds, 
$$
\int_{\xT^d}\eta(t,x)\dx=0 \quad\text{and}\quad \inf_{x\in\xT^d}\eta(t,x)>-\mez h.
$$
In the infinite depth case $h=+\infty$, the last condition is automatically satisfied.
\end{enumerate}
\end{defi}
\begin{rema}
$i)$ The reason to assume that $s>2+d/2 $ \, is that, under this assumption we have 
$\nabla_{x,y}\phi\in  C^1(\overline{\Omega(t)})$ for all time $t$. 
This elementary result (see Proposition~$2.2$ in \cite{Boundary}) will be sufficient  
to rigorously justify 
all integrations by parts arguments.

$ii)$ We can assume without loss of generality that the mean value of $\eta$ is $0$ since this property is propagated in time (see Corollary~\ref{eta1}).

$iii)$ Such regular solutions are known to exist. Indeed, the local well-posedness for the Cauchy problem with initial data in Sobolev spaces has been extensively studied; see 
\cite{Nalimov,Yosihara,Craig1985,BG,WuInvent,AmMa,LindbladAnnals,CS,LannesJAMS,IguchiCPDE,SZ,ABZ3,Kukavica-3,LannesLivre,BD-2018,dPARMA,AiIT-2019,Ayman2}.
\end{rema}

\subsection{Virial theorem} Our first main result is the following

\begin{theo}[Virial theorem]\label{T:virial} Let 
$ g\in \xR$, and consider a 
regular solution $(\eta,\psi)$ to the water-wave system defined on the time interval $[0,T]$. 

$i)$ Assume that $h<+\infty$. For any time $t\in [0,T)$, there holds
\begin{equation}\label{MI1}
\begin{aligned}
\mez \fract
\int_{\xT^d} \eta(t,x)\psi(t,x)\dx &= \iint_{\Omega(t)} \Big(\frac{3}{4}(\partial_y \phi)^2  + \frac{1}{4} \vert \nabla_x \phi\vert^2\Big)(t,x,y)\dydx\\
&\quad- \frac{g}{2}\int_{\xT^d}\eta(t,x)^2\dx+  \frac{h}{4}\int_{\xT^d} \vert \nabla_x \phi(t, x, -h)\vert^2\dx. 
\end{aligned}
\end{equation}
$ii)$ In the infinite depth case $h=+\infty$, the previous identity simplifies to
\begin{equation}\label{MI2}
\begin{aligned}
\mez \fract
\int_{\xT^d} \eta(t,x)\psi(t,x)\dx   = \iint_{\Omega(t)} \Big(\frac{3}{4}(\partial_y \phi)^2  &+ \frac{1}{4} \vert \nabla_x \phi\vert^2\Big) (t,x,y)\dydx\\
&  - \frac{g}{2}\int_{\xT^d}\eta(t,x)^2\dx .
 \end{aligned}
\end{equation}
\end{theo}
\begin{rema}
$i)$ We considered the case where the variable $x$ lies in $\xT^d$ in order to be able to apply our results to Stokes waves or standing waves. However, 
the same results hold when $\xT^d$ is replaced by $\xR^d$, with exactly the same proof.

$ii)$ We shall rigorously justify below the fact that, when the depth of the fluid is very large, we may without sensible error suppose $h$ to be infinite. Indeed  we prove in Proposition \ref{fond-inf1} that the term $h\int_{\xT^d} \vert \nabla_x \phi(t, x, -h)\vert^2\dx$ tends to zero when $h$ goes to $+\infty.$

$iii)$ The study of integral identities for water waves has attracted a lot of attention in fluid mechanics (see Longuet-Higgins~\cite{LH1974}, 
Longuet-Higgins and Fenton~\cite{LHF1974}, Benjamin and Olver~\cite{BO}, Clamond~\cite{Clamond2018}). 
In the case $d=1$, in infinite depth ($h=+\infty$) and for progressive wave only, 
similar computations were made by Starr~\cite{Starr1947} and 
Platzamn~\cite{Platzman1947}. 
\end{rema}

We will deduce from the previous identity a result about the equipartition of energy when one averages in time.
\begin{nota}
Given a function $f=f(t)$ and $T>0$ we denote by $\langle f\rangle_T$ the average over the time interval $[0,T]$: 
$$
\langle f\rangle_T = \frac{1}{T}\int_0^T f(t)\dt.
$$
\end{nota}
Recall that we denote by $E_k$, $E_p$ and $E$ 
the kinetic, potential and total energies given by \e{E1}--\e{E3}. 
Let us also introduce 
the following modified kinetic energy
\be\label{defi:ME}
\ME(t)=\iint_{\Omega(t)} \left(\frac{3}{4}(\partial_y \phi)^2  + \frac{1}{4} \vert \nabla_x \phi\vert^2\right)(t,x,y)\dydx,
\ee
as well as the boundary energy at the bottom:
\be\label{defi:Boundary}
\Boundary(t)=\frac{h}{4}\int_{\xT^d} \vert \nabla_x \phi(t, x, -h)\vert^2\dx.
\ee
\begin{coro}[Equipartition of energy] \label{C:virial}
\begin{enumerate}[i)]
\item\label{C:viriali} {\em Progressive wave :}  
assume that $\eta$ and $\psi$ are of the form
$$
\eta(t,x)=\underline{\eta}(x-ct),\quad \psi(t,x)=\underline{\psi}(x-ct)+f(t)\quad\text{with }
\underline{\eta},\underline{\psi}\in H^s(\xT^d) \text{ and }c\in\xR^d.
$$
Then the energies $\ME$, $E_p$ and $\Boundary$ 
are time independent and moreover
$$
\ME-E_p+\Boundary =0.
$$

\item\label{C:virialii} {\em Periodic solutions : } if $t\mapsto \eta(t,x)$ and $t\mapsto \nabla_{x,y}\phi(t,x,y)$ are $T$-periodic 
then
$$
\langle \ME+\Boundary\rangle_T=\langle E_p\rangle_T.
$$

\item\label{C:virialiii} {\em Arbitrary solutions :}
there exists a constant $C>0$ such that, for any regular solution defined on the time interval $[0,T]$,
\be\label{virialestimate}
\la \langle \ME + \Boundary- E_p\rangle_T \ra\le \frac{CQ}{T}E\quad\text{with}\quad
Q\defn 1+\lA \nabla_x \eta\rA_{L^\infty([0,T]\times\xT^d)},
\ee
where recall that $E=E_k+E_p$ denotes the total energy (which is constant).

\item {\em Infinite depth :} if $h=+\infty$   all the above identities hold with $\Boundary$ replaced by $0$.
\end{enumerate}
\end{coro}
\begin{rema}\label{R:1.7}
$i)$ The first case about progressive waves applies for instance to study Stokes waves. 
The same result applies also to the case where $\xT^d$ is replaced by $\xR^d$, and hence to study solitary water waves.

$ii)$ A word of caution is in order regarding solutions that are periodic in time. As a by-product of our analysis, 
we will see that if $h<+\infty$ then there is no non-trivial solution $(\eta,\psi)$ such that $\eta$ and $\psi$ are periodic in time.
\end{rema}

\subsection{Higher order virial theorems}\label{S:1.4}

In this paragraph, we state two additional virial type identities which involve higher order energies, which means coercive quantities controlling higher order derivatives of $\eta$ and $\psi$. This kind of identities could be useful to test numerical codes. 
We shall work with suitable derivatives. Namely, 
guided by the analysis in~\cite{ABZ3}, we shall work with  
the horizontal and vertical traces of the velocity 
on the free boundary,  together with the Taylor coefficient, that is the quantities:
$$
B= (\partial_y \phi)\arrowvert_{y=\eta},\quad 
V = (\nabla_x \phi)\arrowvert_{y=\eta},\quad \ma=-(\partial_y P)\arrowvert_{y=\eta}.
$$ 
As recalled in the appendix (see Lemma~\ref{L:31}), there holds
\begin{equation}\label{defi:BVbis}
B= \frac{\nabla_x \eta \cdot\nabla_x \psi+ G(\eta)\psi}{1+|\nabla_x \eta|^2},
\qquad
V=\nabla_x \psi -B \nabla_x\eta.
\end{equation}
Recall also that the Taylor coefficient $\ma$ 
can be defined in terms of 
$\eta,\psi$ only (see \cite{Bertinoro} or 
Definition~$1.5$ in \cite{ABZ3}). 

We begin with the following result.
\begin{prop}\label{P:VAC1}
Let $g\in \xR$, and consider a 
regular solution $(\eta,\psi)$ to the water-wave system defined on the time interval $[0,T]$. 

$i)$ Assume that $h<+\infty$. For any time $t\in [0,T)$  there holds
\be\label{vac:i}
\begin{aligned}
\mez \fractt\int_{\xT^d}\eta^2(t,x)\dx&=
\int_{\xT^d}\frac{\gamma}{2}\big(B^2+\la V\ra^2\big)(t,x)\dx\\ 
&\quad -g\int_{\xT^d}(\eta G(\eta)\eta) (t,x)\dx
  -\mez\int_{\xT^d} \vert \nabla_x \phi(t,x,-h)\vert^2\dx,
\end{aligned}
\ee
where $\gamma$ is a non-negative 
coefficient given by
\be\label{defi:gamma}
\gamma=1-G(\eta)\eta.
\ee
$ii)$ In the infinite depth case  $(h=+\infty)$  the previous identity simplifies to
\be\label{vac:ii}
\begin{aligned}
\mez \fractt\int_{\xT^d}\eta^2(t,x)\dx&=
\int_{\xT^d}\frac{\gamma}{2}\big(B^2+\la V\ra^2\big)(t,x)\dx\\
&\quad  -g\int_{\xT^d}(\eta G(\eta)\eta) (t,x)\dx.
\end{aligned}
\ee
$iii)$ Moreover, $B^2+\la V\ra^2$ can be written under the form (for any $h$),
$$
B^2+\la V\ra^2=\frac{(G(\eta)\psi)^2+\la\nabla_x\psi\ra^2+\left(\la \nabla_x\eta\ra^2\la\nabla_x\psi\ra^2-(\nabla_x\eta\cdot\nabla_x\psi)^2\right)}{1+\la\nabla_x\eta\ra^2}.
$$
\end{prop}
As a corollary, we will deduce a coercive estimate controlling the velocity on the free surface in terms of the $L^\infty$-norm of $\eta$.
\begin{coro}[A coercive estimate]\label{coro:C.1.9}
Let $g>0$ and assume that $h=+\infty$. Consider a 
regular solution $(\eta,\psi)$ to the water-wave system defined on the time interval $[0,T]$ and set $M\defn \lA \eta\rA_{L^\infty([0,T]\times\xT^d)}$. Then, 
\be\label{esti:vac7}
\frac{1}{T}\int_0^T\int_{\xT^d}\frac{\gamma}{2}\big(B^2+\la V\ra^2\big)(t,x)\dx\dt\le 
\frac{4\sqrt{M}\sqrt{E}}{T}
+4M,
\ee
where $\gamma$ is given by \e{defi:gamma} and 
$E=E_k+E_p$ denotes the total energy (which is constant). 

Moreover, if $t\mapsto \eta(t,x)$ and $t\mapsto \nabla_{x,y}\phi(t,x,y)$ are $T$-periodic, then the previous inequality simplifies to
\be\label{esti:vac17}
\frac{1}{T}\int_0^T\int_{\xT^d}\frac{\gamma}{2}\big(B^2+\la V\ra^2\big)(t,x)\dx\dt\le 4M.
\ee
\end{coro}

Let us now compare~\e{vac:i} with the main identity derived by Longuet-Higgins in his celebrated 
paper~\cite{LH1974}, which we recall.

\begin{prop}[\cite{LH1974}]
Let $g\in \xR$ and $d\ge 1$. Consider a 
regular solution $(\eta,\psi)$ to the water-wave system defined on the time interval $[0,T]$. Assume that $h<+\infty$. For any time $t\in [0,T)$, there holds
\be\label{LHintro}
\mez \fractt\int_{\xT^d}\eta^2(t,x)\dx=\int_{\xT^d}P(t,x,-h)\dx-g\la \Omega(0)\ra,
\ee
where $P$ is the pressure as given by \e{t5} where $\phi$ is the  harmonic extension of $\psi$.
\end{prop}
\begin{rema}We give another elementary proof of this identity in Section~\ref{S:4.2}. 
Notice that the identity~\e{LHintro} is discussed by Benjamin and Olver, see equation~$(6.14)$ in \cite{BO}, albeit only in the finite depth case. The possible ambiguities that occur when one tries to extend \e{LHintro} to the infinite depth case are also discussed by Benjamin and Olver (see \cite[page 175]{BO}). In sharp contrast, observe that Proposition~\ref{P:VAC1} holds 
in finite or infinite depth. Moreover we have a uniform result in that 
the identity \e{vac:ii} follows from \e{vac:i} by letting $h$ goes to $+\infty$. 
\end{rema}

Then, by comparing the two previous propositions, we shall deduce the following additional virial identity involving the trace of the velocity potential at the bottom.
\begin{coro}\label{C:VAC2}
Let $g\in \xR$, and consider a 
regular solution $(\eta,\psi)$ to the water-wave system defined on the time interval $[0,T]$. 
Assume that $h<+\infty$. For any time $t\in [0,T)$  there holds, with $\gamma$   given by \eqref{defi:gamma},
$$
\begin{aligned}
\fract\int_{\xT^d}\phi(t,x,-h)\dx&=
g\int_{\xT^d}(\eta G(\eta)\eta) (t,x)\dx\\
&\quad -\int_{\xT^d}\frac{\gamma}{2}\big(B^2+\la V\ra^2\big)(t,x)\dx.
\end{aligned}
$$
 \end{coro}

We now turn to another kind of virial identities. 
To introduce this part, we begin by explaining that, loosely speaking, the first 
integral in the right-hand 
side of \e{vac:i} controls the $\dot{H}^1$-norm of $\psi$ 
while the second integral controls 
the $\dot{H}^{1/2}$-norm of $\eta$. So 
this corresponds to a virial type identity 
for quantities involving 
energies $1/2$-derivative more regular than the 
ones which appear in 
Theorem~\ref{T:virial}. The next result corresponds to a virial type identity 
for quantities involving 
energies $1$-derivative more regular than those 
appearing in 
Theorem~\ref{T:virial}. 
\begin{prop}\label{P:VAC3}
Let $g\in \xR$ and $h\in (0,+\infty]$. 
For any regular solution $(\eta,\psi)$, the following two properties are verified:
\be\label{virial:ordre1V}
\fract\int_{\xT^d}\nabla_x\eta\cdot V\dx
=\int_{\xT^d}V\cdot G(\eta)V\dx
-\int_{\xT^d}a\la \nabla_x\eta\ra^2\dx,
\ee
(where $V\cdot G(\eta)V=\sum_{j=1}^d V_jG(\eta)V_j$). 
Moreover, if $h=+\infty$, there holds
\be\label{virial:ordre1B}
\fract\int_{\xT^d}B\dx
=\int_{\xT^d}B G(\eta)B\dx
-\int_{\xT^d}(a-g)\dx,
\ee
where $a$ is the Taylor coefficient.
\end{prop}
\begin{rema}
$i)$ Since, $G(\eta)\psi=B-V\cdot\nabla_x \eta$ and since $\int_{\xT^d}G(\eta)\psi\dx=0$, we have,
$$
\int_{\xT^d}B\dx=\int_{\xT^d}\nabla_x \eta\cdot V\dx.
$$
$ii)$ Since $B(t,x) = \partial_y \phi(t,x,\eta(t,x))$ and $\int_{\xT^d}BG(\eta)B\dx\ge 0$  it follows from \e{virial:ordre1B} that,  for any solution such that $t \mapsto \nabla_{x,y}\phi(t,x,y)$ and $t \mapsto \eta(t,x)$ are periodic   with period $T$  there holds,
$$
\int_0^T\int_{\xT^d}(a(t,x)-g)\dx\dt\ge 0.
$$
Which means that, in average, the Taylor coefficient $a$ is greater than the acceleration of gravity $g$. 
\end{rema}
\subsection{Rayleigh-Taylor instability} 
For other dispersive equations, virial type identities have been used 
to study blow-up phenomena 
(see Levine~\cite{Levine,Levine2},  Glassey \cite{Glassey-blowup}, 
Keel-Tao~\cite{Keel-Tao-1999}, Kenig-Merle~\cite{Kenig-Merle-2008}, 
Vega and Visciglia~\cite{VegaVisciglia2008JFA} and the references therein). 
We will see that, for the water-wave problem, the 
previous virial identity can be applied to study the Rayleigh-Taylor 
instability (named after Rayleigh~\cite{Rayleigh1879} and Taylor~\cite{TaylorG}). 

The Rayleigh-Taylor instability is an instability 
of an interface between two fluids, 
which occurs when a heavy fluid layer is supported by a light one (see Sharp~\cite{Sharp1984}). The Rayleigh-Taylor instability can be divided into three 
stages: (1) growth, (2) bubble formation and (3) rising air columns (with mixing). 
We will be concerned by the analysis of growth and bubble formation. 

The mathematical analysis of this instability 
has attracted much attention in the case of two fluids: 
We refer to the study of Bardos and Lannes~\cite{BL} 
and the results of Kamotski and Lebeau~\cite{KL2005}, Lafitte~\cite{Lafitte2008}, 
Wilke~\cite{Milke2017}, Pr\"{u}ss, Simonett and Wilke~\cite{PSW2019}, 
Gebhard, Kolumb\'{a}n and Sz\'{e}kelyhidi~\cite{GBS2021}. 
Here, we are interested in the water-air interface problem, which is the case with a single fluid since the density of air is considered as negligible. 
Specifically, we will consider exactly the case considered by Kull~\cite{Kull1991} 
in his seminal paper on this topic (to which we refer for 
a thorough physical description of the phenomena considered). 

The vertical axis 
being directed upwards, the assumption that 
the water is located above the air 
is equivalent to assuming that the acceleration of gravity 
is non-positive, that is $g<0$. In fact we will consider the 
case $g\le 0$. When $g<0$, for linearized equations, it is 
elementary to see that the problem is microlocally 
ill-posed on the Sobolev spaces. 
For the full model, i.e.\ the incompressible Euler equation with 
free surface, to our knowledge the only ill-posed result 
is due to Sideris~\cite{SiderisJDE14}. For the case of a bounded three-dimensional region surrounded by vacuum, 
he proved that the diameter of the region grow linearly in time. Here we will obtain a similar result for a different shape 
of fluid domain (the equations and the assumptions on the domain that we are making here are exactly those made by Kull~\cite{Kull1991}). 
We will also consider the case where $g<0$. Notice that, for our setting, 
when $g=0$ the Cauchy problem is well-posed 
for initial data with finite regularity (see Agrawal~\cite{Agrawal2022}) and 
when $g\le 0$, the existence of regular solutions 
is known for analytic initial data (see~\cite{Analytic} and the references therein). 

We will prove that, for {\emph {any}} possible solutions, the maximum slope of the free surface must grow like $Ct^{a}$ for some $a>0$, thus justifying the formation of bubbles for the free-surface Rayleigh-Taylor instability, for 
any non-zero initial data.

For the sake of readability we consider the case $g=0$ and $g<0$ in two distinct statements.

\begin{prop}\label{P:RT0}
Assume that $g=0$ and $h\in [1,+\infty]$. 
There exists a constant $C>0$ such that, for all regular solution $(\eta,\psi)$ to the water-wave system defined on the time interval $[0,T]$, there holds
\be\label{RT:0}
\lA \eta(t)\rA_{L^2}(1+\lA \nabla_x\eta(t)\rA_{L^\infty})^\mez\ge \frac{c}{\sqrt{E}} \Big(
E t+\int_{\xT^d}\eta(0,x)\psi(0,x)\dx \Big),
\ee
where, 
$E=\mez \iint_{\Omega(0)}\la \nabla_{x,y}\phi(0,x,y)\ra^2\dydx$.
\end{prop}
\begin{rema}
$i)$ The same result holds with $\lA \eta(t)\rA_{L^2}$ replaced by $\lA \eta(t)\rA_{\dot{H}^{-1/2}}$. 

$ii)$ Since $\int\eta(t,x)\dx=0$, it is easily seen that (see~\e{poinc2})
$$
\lA \eta\rA_{L^2}\le C\lA \eta\rA_{L^\infty}\le C'\lA \nabla_x\eta\rA_{L^\infty},
$$
and hence we deduce from \e{RT:0} a lower bound involving only the slope:
$$
(1+\lA \nabla_x\eta(t)\rA_{L^\infty})^{3/2}\ge \frac{c'}{\sqrt{E}} \Big(
E t+\int_{\xT^d}\eta(0,x)\psi(0,x)\dx \Big).
$$
\end{rema}

Eventually, we extend the previous analysis to the case $g< 0$.
\begin{prop}\label{P:RT0bis}
Assume that $g< 0$ and fix $h\in (0,+\infty]$. 
There exists a constant $C>0$ such that, for all regular solution $(\eta,\psi)$ to the water-wave system defined on the time interval $[0,T]$, there holds
\be\label{RT:g}
C\lA \eta(t)\rA_{L^{2}}\left(1+\lA \nabla_x \eta(t)\rA_{L^\infty}\right)^\mez\left(E+ \frac{|g|}{2}\lA \eta(t)\rA_{L^{2}}^2\right)^\mez\ge |E| t+ \int_{\xT^d} \eta(0,x)\psi(0,x)\dx.
\ee
where $E\in \xR$ is the total energy.
\end{prop}
\begin{rema}
$i)$ 
The total energy is constant in time, but not necessarily 
non-negative when $g<0$. However the quantity 
$E+ \frac{|g|}{2}\lA \eta(t)\rA_{L^{2}}^2$ is non-negative in light of 
the conservation of energy (see \e{RT:4} below); so that the square root is well-defined even for $E<0$.

$ii)$ To illustrate this result, 
assume that $\psi(0,x)=0$. Then the energy is given by $E=-\frac{|g|}{2}\int \eta(0,x)^2\dx<0$ and the previous 
statement implies that
$$
\frac{C^2|g|}{2}\lA \eta(t)\rA_{L^{2}}^4\left(1+\lA \nabla_x \eta(t)\rA_{L^\infty}\right)
\ge |E|^2 t^2.
$$
\end{rema}

\subsection{Plan of the paper}
Our main virial-type identity is established in Section~\ref{S:2}. We then deduce the results about the equipartition 
of energy in Section~\ref{S:3} and the ones about the Rayleigh instability in Section~\ref{S:4}. 
In Section~\ref{S:5} we illustrate the virial identity for the key example of 
standing water-waves. In the Appendix we recall various identities and also, 
in Appendix~\ref{Appendix:C}, we prove several trace inequalities of independent interest which are used to 
control the trace of harmonic functions either on the free surface or at the bottom.

\section{The virial theorem}\label{S:2}

In this section, we prove the virial Theorem~\ref{T:virial} and 
its Corollary~\ref{C:virial}. 

Recall the following notations for the various energies:
\begin{align*}
E_k(t)&=\mez\iint_{\Omega(t)}\la \nabla_{x,y}\phi(t,x,y)\ra^2\dydx,\quad E_p(t) =\frac{g}{2}\int_{\xT^d}\eta(t,x)^2\dx,\\
E&=E_k+E_p,\\
\ME(t)&=\iint_{\Omega(t)} \left(\frac{3}{4}(\partial_y \phi)^2  + \frac{1}{4} \vert \nabla_x \phi\vert^2\right)(t,x,y)\dydx,\\ \Boundary(t)&=\frac{h}{4}\int_{\xT^d} \vert \nabla_x \phi(t, x, -h)\vert^2\dx. 
\end{align*}

\subsection{Proof of the virial theorem in the finite depth case}
Assume that $h<+\infty$. 
We want to prove that, for all 
regular solution $(\eta,\psi)\in C^0([0,T], H^s(\xT^d)\times H^s(\xT^d))$ to the water-wave system  and for all 
time $t\in [0,T)$  there holds
\begin{equation}
\mez \fract
\int_{\xT^d} \eta(t,x)\psi(t,x)\dx = \ME(t)
-E_p(t)+ \Boundary(t).
\end{equation}

Putting for shortness
\begin{equation}\label{mez-3mez0}
(1) = \fract\int_{\xT^d} \eta(t,x)\, \psi(t,x)\dx,
 \end{equation}
we see that, directly from the formulation~\e{systemT} of the water-wave problem, we have
 \begin{align*}
(1)&= \int_{\xT^d}\partial_t \eta(t,x) \psi(t,x)\dx +\int_{\xT^d}\eta(t,x)\, \partial_t \psi(t,x)\dx\\
 &= \int_{\xT^d} (\psi G(\eta)\psi )(t,x)\dx  
  - g  \int_{\xT^d} \eta(t,x)^2\dx  -  \int_{\xT^d} \eta(t,x) N(\eta,\psi)(t,x)\dx.
 \end{align*}
Using the divergence theorem (see~\e{positivityDN}), we obtain at once
$$
\int_{\xT^d} (\psi G(\eta)\psi )(t,x)\dx=
\iint_{\Omega(t)}\big((\partial_y \phi)^2 + \vert \nabla_x \phi\vert^2\big)(t,x,y)\dydx .
$$
It follows that 
\begin{equation}\label{mez-3mez1}
\begin{aligned}
(1) =\iint_{\Omega(t)} \big((\partial_y \phi)^2 + \vert \nabla_x \phi\vert^2\big)(t,x,y)\dydx  &-2 E_p(t)\\
& - \int_{\xT^d} \eta(t,x) N(\eta,\psi)(t,x)\dx.
\end{aligned}
\end{equation}
Let us compute the second integral in the right-hand side. To do so, we use a variant of 
the multiplier method used by Rellich to 
derive its celebrated  identity \footnote{The Rellich inequality shows the equivalence between 
the $L^2$-norm of the tangential derivatives 
and the $L^2$-norm of the normal derivative. It plays a key role 
in the study of boundary value problems in Lipschitz domains (see Ne{\v c}as~\cite[Chapter 5]{Necas}, 
Brown~\cite{Brown1994} or McLean~\cite[Theorem 4.24]{McLean}).} involving the tangential and normal derivatives of an harmonic function. Namely 
the idea is to introduce the integral
$$
(2)=  \iint_{\Omega(t)} \Delta_{x,y}\phi(t,x,y)(y\partial_y \phi)(t,x,y)\dydx.
$$ 
Since $\phi$ is harmonic, we have of course $(2)=0$. 
On the other hand, since $\cnx_x(aF) = a \cnx_xF 
+ \nabla_x a \cdot F$, 
by setting $I:=(\Delta_{x,y} \phi) (y\partial_y \phi)$, we have
\begin{align*}
 &I= \mez y\partial_y\big((\partial_y \phi)^2\big) + \cnx_x \big((y\partial_y \phi)\nabla_x \phi\big) - \mez y \partial_y \big(\vert \nabla_x \phi\vert^2\big),\\
 &= \mez \partial_y \big[y \big((\partial_y \phi)^2 -\vert \nabla_x \phi\vert^2\big)\big] - \mez \big((\partial_y \phi)^2 -\vert \nabla_x \phi\vert^2\big)+  \cnx_x \big((y\partial_y \phi)\nabla_x \phi\big).
\end{align*}
Now introduce the following traces 
$$
B=\partial_y \phi\arrowvert_{y=\eta}, \quad
V =(\nabla_x \phi)\arrowvert_{y=\eta}.$$
Since $\partial_y \phi\arrowvert_{y = -h} = 0$  we have,
\begin{equation}\label{ipp}
\begin{aligned}
  \mez\int_{\xT^d}\int_{-h}^{\eta(x)}   \partial_y \big[y \big((\partial_y \phi)^2 -\vert \nabla_x \phi\vert^2\big)\big]\dx &= \mez \int_{\xT^d} \eta(t,x)\big( B^2- \vert V \vert^2\big)(t,x)\dx\\
  &\quad- \frac{h}{2}\int_{\xT^d} \vert \nabla_x \phi(t, x, -h)\vert^2\dx.
\end{aligned}
\end{equation}
By combining the previous identities, we deduce that
\begin{align*}
  (2)&= \mez \int_{\xT^d} \eta(t,x)\big( B^2- \vert V \vert^2\big)(t,x)\dx - \frac{h}{2}\int_{\xT^d} \vert \nabla_x \phi(t, x, -h)\vert^2\dx\\
  & \quad- \mez \iint_{\Omega(t)} \big((\partial_y \phi)^2 -\vert \nabla_x \phi\vert^2\big)(t,x,y)\dydx + \iint_{\Omega(t)}   \cnx_x \big((y\partial_y \phi)\nabla_x \phi\big)\dydx.
    \end{align*}
On the other hand  it is easily verified that
\begin{align*}
0&=\int_{\xT^d}  \cnx_x \int_{-h}^{\eta(t,x)}((y\partial_y \phi)\nabla_x \phi)(t,x,y)\dydx\\
&=   \iint_{\Omega(t)}   \cnx_x \big((y\partial_y \phi)\nabla_x \phi\big(t,x,y)\big)\dydx \\
&\quad + \int_{\xT^d}(y\partial_y \phi)(\nabla_x\eta\cdot \nabla_x\phi)\arrowvert_{y = \eta(t,x)}\dx
     \end{align*}
which in turn implies that
$$
\iint_{\Omega(t)}   \cnx_x \big((y\partial_y \phi)\nabla_x \phi\big(t,x,y)\big)\dydx = -\int_{\xT^d}  \eta(t,x)\big( B( V\cdot  \nabla_x \eta)\big)(t,x)\dx.
$$
Consequently, we get
\begin{align*}
(2)
&=- \mez \iint_{\Omega(t)} \big((\partial_y \phi)^2 -\vert \nabla_x \phi\vert^2\big)(t,x,y)\dydx  \\
&\quad + \int_{\xT^d} \eta(t,x) \left[ \mez B^2 - \mez \vert V \vert^2 - B( V\cdot  \nabla_x \eta)\right](t,x)\dx\\
&\quad - \frac{h}{2}\int_{\xT^d} \vert \nabla_x \phi(t, x, -h)\vert^2\dx.
  \end{align*}
Now, by definition of $N$ (see~\e{t90}), we have,
$$
N= \mez \vert V \vert^2 -\mez B^2+ B( V\cdot  \nabla_x \eta).
$$
Therefore,
\begin{align*}
(2)
&=  - \mez \iint_{\Omega(t)} \big((\partial_y \phi)^2 -\vert \nabla_x \phi\vert^2\big)(t,x,y)\dydx \\
&\quad -\int_{\xT^d} \eta(t,x)N(t,x)\dx - \frac{h}{2}\int_{\xT^d} \vert \nabla_x \phi(t, x, -h)\vert^2\dx.
\end{align*}
Remembering that $(2) = 0$, we infer that,
\begin{equation}\label{mez-3mez2}
\begin{aligned}
- \int_{\xT^d} \eta(t,x) N(\eta,\psi)(t,x)\dx &=  \mez \iint_{\Omega(t)} \big((\partial_y \phi)^2 -\vert \nabla_x \phi\vert^2\big)(t,x,y)\dydx\\
&\quad+ \frac{h}{2}\int_{\xT^d} \vert \nabla_x \phi(t, x, -h)\vert^2\dx .
\end{aligned}
\end{equation}
By combining \eqref{mez-3mez0}, \eqref{mez-3mez1} and \eqref{mez-3mez2}, we end up with
\begin{equation}\label{mez-3mez3}
\begin{aligned}
\fract \int_{\xT^d} \eta(t,x)\psi(t,x)\dx = \iint_{\Omega(t)} \Big(\frac{3}{2}(\partial_y \phi)^2  &+ \mez \vert \nabla_x \phi\vert^2\Big)(t,x,y)\dydx  -2E_p(t)\\
& + \frac{h}{2}\int_{\xT^d} \vert \nabla_x \phi(t, x, -h)\vert^2\dx,
\end{aligned}
\end{equation}
equivalent to the wanted result.

\subsection{Proof of the virial theorem in the infinite depth case} 
The proof of statement $ii)$ in Theorem~\ref{T:virial} is similar to the above proof. 
However in order to be able to perform the same computations   as in the finite depth case we need (see \eqref{ipp}) to rigorously justify that
\be\label{decay:var2}
\lim_{y\to-\infty}\vert y \vert \int_{\xT^{d}}\la \nabla_{x,y}\phi(t,x,y)\ra^2\dx=0.
\ee
The proof of \e{decay:var2} is the 
subject of this subsection.

In \eqref{decay:var2} the variable $t$ is fixed, so we will skip it in what follows.

Recall that 
$\eta$ is bounded by hypothesis. 
We can assume 
without loss of generality that 
$\inf_{x\in\xT^d}\eta(t,x)>-1$. 
In particular the hyperplane 
of equation $\{y=-1\}$ 
is contained in the domain $\Omega$. 
By a uniqueness result, 
this will allow us to say 
that $\phi$ 
coincides in the half-space $\{y<-1\}$ with 
the harmonic extension (denoted by $\theta$) of 
its trace on $\{y=-1\}$. 
The idea is that 
$\theta$ can be studied thanks to 
the 
Fourier transform (since the 
problem defining $\theta$ is translation invariant in $x$) 
and will see that 
the desired result follows easily from Plancherel's theorem. 

Let us now proceed 
to the details. Let $\theta$ 
be the variational solution of the problem
\begin{equation}\label{eq-theta}
\Delta_{x,y} \theta = 0  \quad \text{ for } -\infty<y<-1, 
\quad \theta\arrowvert_{y=-1} = \phi(x,-1).
\end{equation}

It satisfies
\begin{equation}\label{varia}
\int_{-\infty}^{-1} \int_{\xT^d} \vert \nabla_{x,y} \theta(x,y)\vert^2\dx\dy <+\infty.
\end{equation}
By virtue of the Plancherel and Fubini theorems, we deduce that
$$
\sum_{\xi \in \xZ^d} \int_{-\infty}^{-1} \vert \xi \vert^2 \vert \widehat{\theta}(\xi,y)\vert^2\dy <+\infty.
$$
Therefore, for all  $\xi \in \xZ^d$ with  $\vert \xi \vert \geq 1$, 
we have
 \begin{equation}\label{fini}
  \int_{-\infty}^{-1} \vert \xi \vert^2 \vert \widehat{\theta}(\xi,y)\vert^2\dy <+\infty.
  \end{equation}
Performing a  Fourier transform in the    $x$ variable we get from    \eqref{eq-theta},  
$$(\partial_y^2 - \vert \xi \vert^2)\widehat{\theta} = 0 \text{ in } -\infty<y<-1, \quad \widehat{\theta}(\xi, -1) = \widehat{\phi}(\xi, -1).$$
Then,
\begin{equation}\label{hat-theta}
\begin{aligned}
  \widehat{\theta}(\xi,y) &= C_1(\xi)e^{y \vert \xi\vert} +C_2(\xi)e^{-y \vert \xi\vert} ,\\
   \widehat{\phi}(\xi, -1)&= C_1(\xi)e^{- \vert \xi\vert} +C_2(\xi)e^{  \vert \xi\vert}.
 \end{aligned}
 \end{equation}
It is easily verified that 
the 
condition \eqref{fini} implies that $C_2(\xi) = 0$ for all $\xi \neq 0$, which implies that
$$\widehat{\theta}(\xi, y) =  e^{(1-\vert y\vert) \vert \xi \vert}\widehat{\phi}(\xi, -1), \quad \xi \neq 0,$$
which implies that  for $y<0,$
$$
\partial_y \widehat{\theta}(\xi, y) = \vert \xi \vert \widehat{\theta}(\xi, y).
$$
For $\vert y \vert \geq  3 $ we have $1-\vert y\vert \leq  - \mez \vert y\vert -\mez$ 
so that,
$$\vert \widehat{\theta}(\xi, y)\vert \leq e^{-\mez \vert \xi \vert} 
e^{- \mez \vert y \vert \vert\xi \vert}\vert \widehat{\phi}(\xi, -1)\vert \leq  e^{-\mez \vert \xi \vert}e^{- \mez \vert y  }\vert \widehat{\phi}(\xi, -1)\vert,
$$
since $\xi \in \xZ^d, \xi \neq 0$ imply $\vert \xi \vert \geq 1.$ Therefore,
$$
\int_{\xT^d} \vert \nabla_{x,y} \theta(x,y)\vert^2\dx \leq C_d \sum_{\xi\in \xZ^d \setminus 0 }\vert \xi\vert^2 \vert \widehat{\theta}(\xi,y)\vert^2 \leq C_d e ^{-\vert y \vert} \sum_{\xi\in \xZ^d \setminus 0 }\vert \xi\vert^2 e^{-\vert \xi \vert}\vert \widehat{\phi}(\xi, -1)\vert^2.$$
On the other hand, since $\phi$  is a   $C^\infty$ fonction in $\{(x,y): x \in \xT^d, -\infty<y<-\mez\}$ we have,
$\vert \widehat{\phi}(\xi, -1)\vert \leq C\Vert \phi(\cdot, -1)\Vert_{L^1(\xT^d)}$ so that eventually,
$$ \int_{\xT^d} \vert \nabla_{x,y} \theta(x,y)\vert^2\dx  \leq C'_d e ^{-\vert y \vert} \Vert \phi(\cdot, -1)\Vert^2_{L^1(\xT^d)}.$$
Then \eqref{decay:var2} follows from the fact that $\theta = \phi$ for $y<-1.$

\section{Equipartition of energy}\label{S:3}

In this section we prove Corollary~\ref{C:virial}.

\subsection{Progressive waves}
Statement~$\ref{C:viriali})$ in Corollary~\ref{C:virial} is obvious.

\subsection{Periodic solutions}
The proof of statement~$\ref{C:virialii})$ in Corollary~\ref{C:virial} relies on the 
following straightforward 
result.
 
\begin{lemm}\label{Etilde1}
Consider a regular solution $(\eta, \psi)$ such that,
\begin{enumerate}
\item\label{p:(ii)bis} $t\mapsto \eta(t,x)$  is T-periodic,
\item\label{p:(iii)bis} $t\mapsto \nabla_{x,y} \phi(t,x,y)$ is T-periodic,
\end{enumerate}
where recall that $\phi$ is the harmonic extension of $\psi$. Then the function, 
$$
t \mapsto \theta(t) =  \int_{\xT^d} \eta(t,x) \psi(t,x)\dx  \text{ is } T\text{-periodic}.
$$
 \end{lemm}
 \begin{proof}
By assumption,
$$
\nabla_{x,y}( \phi(t, x,y) - \phi(t+ T,  x,y))= 0, \quad \forall x,y\in \xT^d
$$
and hence, there is a function $\alpha$ depending only on time such that,
$$
\phi(t, x,y) - \phi(t+ T,  x,y) = \alpha(t).
$$
Now, using assumption $\ref{p:(ii)bis})$, we see that,
\begin{align*}
 \psi(t + T,x) &=   \phi(t+ T, x,\eta(t+ T,x)) = \phi(t+ T, x,\eta(t,x)),\\
 &=  \phi(t, x,\eta(t,x)) -\alpha(t) = \psi(t,x)-\alpha(t).
 \end{align*}
Set,
$$
\beta(t)= \frac{1}{\vert \xT^d\vert} \int_{\xT^d}\psi(t,x)\dx,\quad 
\widetilde{\psi}(t,x)\dx = \psi(t,x) -  \beta(t).
$$
We have, 
$$\beta(t+T) =  \frac{1}{\vert \xT^d\vert} \int_{\xT^d}\psi(t+T,x)\dx =  \frac{1}{\vert \xT^d\vert} \int_{\xT^d}\psi(t,x)\dx -\alpha(t) = \beta(t)-\alpha(t).$$
We then write,
\begin{align*}
 \widetilde{\psi}(t+T,x) &= \psi(t+T,x) - \beta(t+T) = \psi(t,x) -\alpha(t) - (\beta(t)-\alpha(t))= \widetilde{\psi}(t,x).
\end{align*}
Now, as already seen in the proof of Lemma~\ref{est-psi}, using the fact that $\int_{\xT^d}\eta(t,x)\dx=0$ (by assumption on regular solutions, see Definition~\ref{defi:regular}), we have,
$$
\theta(t) =  \int_{\xT^d} \eta(t,x) \widetilde{\psi}(t,x)\dx + \beta(t) \int_{\xT^d} \eta(t,x)\dx = \int_{\xT^d} \eta(t,x) \widetilde{\psi}(t,x)\dx.
$$
This proves that $\theta$ is $T$-periodic, which completes the proof. 
\end{proof}
\begin{coro}\label{Etilde2}
Consider a regular solution $(\eta, \psi)$ such that,
\begin{enumerate}
\item\label{p:(ii)} $t\mapsto \eta(t,x)$  is T-periodic,
\item\label{p:(iii)} $t\mapsto \nabla_{x,y} \phi(t,x,y)$ is T-periodic.
\end{enumerate}
Then,
\begin{multline*}
  \int_0^{T}\iint_{\Omega(t)} \Big(\frac{3}{4}(\partial_y \phi)^2 + \frac{1}{4} \vert \nabla_x \phi\vert^2\Big)(t,x,y)\dydx\dt  -\int_0^{T}E_p(t)\dt\\
  + \frac{h}{4}\int_0^{T}\int_{\xT^d} \vert \nabla_x \phi(t, x, -h)\vert^2\dx\dt  =0.
\end{multline*}
If in addition $h=+\infty$, then this simplifies to
$$
\int_0^{T}\iint_{\Omega(t)} \Big(\frac{3}{4}(\partial_y \phi)^2  + \frac{1}{4} \vert \nabla_x \phi\vert^2\Big)(t,x,y)\dydx\dt  = \int_0^{T}E_p(t)\dt.
$$
\end{coro}
\begin{proof}
The previous lemma implies that,
$$
\int_0^{T}\fract
\int_{\xT^d} \eta(t,x)\psi(t,x)\dx\dt=0.
$$
Hence the Corollary follows immediately from Theorem~\ref{T:virial}.
\end{proof}
To conclude this part, we shall prove an identity which implies that, in finite depth, there are no regular solutions $(\eta,\psi)$ which 
are periodic in time (this does not contradict the previous result which assumes only that $\nabla_{x,y}\phi$ is periodic in time).
\begin{prop}\label{dt-int-psi}
Assume that $h<+\infty$ and consider a regular solution $(\eta,\psi)$ to the water-wave system defined on the time interval $[0,T]$. For all $t \in (0,T),$ 
\begin{equation}\label{eq-dt-int-psi}
\fract\int_{\xT^d} \psi(t,x)\dx = -\mez \int_{\xT^d} \vert \nabla_x  \phi(t,x, -h)\vert^2 \dx.
\end{equation}
In particular, the only regular 
solution  $(\eta,\psi)$ which is periodic in time is of the form $(\eta,\psi)=(0,K)$ for some constant $K$.
\end{prop}
\begin{proof}
Let us first show how \e{eq-dt-int-psi} implies that periodic in time solutions are of the form $(\eta,\psi)=(0,K)$ for some constant $K$. 
If $t\mapsto \psi(t,x)$ is  $T$-periodic then, integrating the equality \eqref{eq-dt-int-psi} between $0$ and $T,$ we conclude that $\nabla_x  \phi(t,x, -h)= 0$ for all $(t,x) \in (0,T)\times \xT^d.$ This implies that $\phi(t,x,-h) = \alpha(t)$ for all $(t,x) \in (0,T)\times \xT^d.$ On the other hand we know that $\phi$ satisfies,
$$\Delta_{x,y} \phi = 0 \text{ in } \, \Omega(t)\quad\text{and}\quad
\partial_y\phi(t,x,-h)=0.$$
 By Holmgren's uniqueness theorem (for fixed $t$) and the analyticity of $\phi$ in $\Omega(t)$ we see that $\phi(t,x,y) = \alpha(t)$ for all  $(x,y)\in \Omega(t).$  So $\psi(t,x) = \phi(t,x,\eta(t,x)) = \alpha(t).$ It follows that $G(\eta)\psi(t,x) = \nabla_x \psi(t,x)=0$ for all $x$ in $\xT^d.$ By equations \eqref{systemT} we deduce easily that $(\eta, \psi) = (0,K)$ where $K$ is a real constant.

We now have to prove \e{eq-dt-int-psi}. 
The proof relies on the following result.
\begin{lemm}\label{intN}
Let $N$ be as defined  in \eqref{t90}. Then, for all $t \in (0,T)$,
$$\int_{\xT^d} N(\eta,\psi)(t,x)\dx = \mez\int_{\xT^d} \vert \nabla_x \phi(t,x,-h)\vert^2\dx.$$
\end{lemm}
\begin{rema}
This identity 
is a version of the classical Rellich identity (see~\cite{A-stab-AnnalsPDE,Agrawal-A}). 
More precisely, the only difference with the classical Rellich identity comes from the boundary term evaluated at the bottom, and this lemma states that this boundary term has a sign.
\end{rema}
\begin{proof}
Here the time variable plays the role of a parameter and we omit it. 
We use again the notations 
$B=\partial_y \phi\arrowvert_{y=\eta}$ and $V =(\nabla_x \phi)\arrowvert_{y=\eta}$. 
Then, by definition of $N$ (see~\e{t90}), we have
$N= \mez \vert V \vert^2 -\mez B^2+ B( V\cdot  \nabla_x \eta)$.

The basic idea to prove a Rellich identity is to multiply the equation $\Delta_{x,y}\phi=0$ by $L\phi$ for some vector field~$L$  transverse to the boundary. 
Here, choosing $L=\partial_y$, 
we write
\begin{equation}\label{int-N0}
\begin{aligned}
 &\iint_\Omega \Delta_{x,y} \phi(x,y) \partial_y \phi(x,y)\dydx=  (1) + (2) =0 \quad\text{with,}\\
 &(1) =\iint_\Omega \partial^2_y \phi(x,y) \partial_y \phi(x,y)\dydx, \quad 
(2) =  \sum_{j=1}^d   \iint_\Omega\partial^2_{x_j}  \phi(x,y) \partial_y \phi(x,y)\dydx.
 \end{aligned}
 \end{equation}
Observe that, since $\partial_y \phi\arrowvert_{y=-h} = 0$, we have
\begin{equation}\label{int-N1}
(1) = \mez  \iint_\Omega \partial_y\big( (\partial_y \phi)^2\big)(x,y)\dydx 
= \mez \int_{\xT^d}  (\partial_y \phi)^2\arrowvert_{y = \eta(x)}\dx.
\end{equation}

We then write,
\begin{equation*} 
\begin{aligned}
   0 & =\int_{\xT^d} \partial_{x_j}\int_{-h}^{\eta(x)} \partial_{x_j}\phi(x,y)\partial_{y}\phi(x,y)\dydx = A_1+ A_2+ A_3,\\
   A_1 &= \int_{\xT^d}\partial_{x_j}\eta(x) 
   (\partial_{x_j}\phi)\arrowvert_{y = \eta(x)}(\partial_{y}\phi)\arrowvert_{y = \eta(x)}\dx =  \int_{\xT^d}\partial_{x_j}\eta(x)V_j(x)B(x)\dx,\\
   A_2 &=  \iint_\Omega\partial^2_{x_j}  \phi(x,y) \partial_y \phi(x,y)\dydx,\\
   A_3&=  \iint_\Omega\partial_{x_j}  \phi(x,y) \partial_{x_j}\partial_y \phi(x,y)\dydx 
   = \mez \iint_\Omega \partial_y \big((\partial_{x_j}  \phi(x,y))^2\big)\dydx.
  \end{aligned}
  \end{equation*}
We deduce that,
  \begin{equation}\label{int-N2}
   (2) = -\int_{\xT^d} B(x) (V(x)\cdot\nabla_x \eta(x))\dx -\mez \iint_\Omega \partial_y \big(\vert \nabla_x  \phi(x,y)\vert^2\big)\dydx.
   \end{equation}
By using \eqref{int-N0},   \eqref{int-N1} and  \eqref{int-N2}, it follows that,
$$
0=  \mez \int_{\xT^d} \Big(\int_{-h}^{\eta(x)} \partial_y\big[ (\partial_y \phi)^2 (x,y) -\vert \nabla_x  \phi(x,y)\vert^2\big] \dy\Big)\dx  -\int_{\xT^d}  B  (V\cdot\nabla_x \eta) \dx.$$
Using again 
$\partial_y\phi(x, -h) = 0$  we infer that, 
$$ \int_{\xT^d} \big[\big(\mez B^2- \mez\vert V\vert^2 - B (V\cdot\nabla_x \eta\big)(x) + \mez\vert \nabla_x  \phi(x, -h)\vert^2\big]\dx =0.
$$
As a consequence, the wanted result follows from the identity \eqref{N=} for $N$.
\end{proof}
The previous lemma immediately implies that,  
$$
\fract\int_{\xT^d} \psi(t,x)\dx =     \int_{\xT^d} \partial_t \psi(t,x)\dx = - g\int_{\xT^d}\eta(t,x)\dx - \int_{\xT^d} N(\eta,\psi)(t,x)\dx.
$$
The Proposition now follows from Proposition~\ref{eta1} and Lemma~\ref{intN}, since a regular solution satisfies $\int_{\xT^s} \eta(t,x) \dx =0.$
\end{proof}

\subsection{The virial estimate}
It remains to prove statement~$\ref{C:virialiii})$ in Corollary~\ref{C:virial}. 

Consider a regular solution $(\eta, \psi)$ defined on the time interval $[0,T]$. 
By integrating the virial identity in time, we have
$$
\langle \ME + \Boundary- E_p\rangle_T=\frac{1}{T}\left(
\int_{\xT^d} \eta(T,x)\psi(T,x)\dx-\int_{\xT^d} \eta(0,x)\psi(0,x)\dx\right),
$$
and hence
$$
\la \langle \ME + \Boundary- E_p\rangle_T\ra
\le \frac{2}{T}\sup_{t\in[0,T]}
\la \int_{\xT^d} \eta(t,x)\psi(t,x)\dx\ra.
$$
So, to prove statement~$\ref{C:virialiii})$, it will suffice to bound 
$\int \eta\psi\dx$. To do so, we will use the following

\begin{lemm}\label{est-psi}
There exists a constant $C>0$ such that, for all $h\ge 1$, for all 
$\sigma\in W^{1,\infty}(\xT^d)$ satisfying $\int_{\xT^d}\sigma(x)\dx=0$  and for all  $f\in H^\mez(\xT^d)$  there holds,
$$
\la \int_{\xT^d}\sigma(x)f(x)\dx\ra\le C
\lA \sigma\rA_{\dot{H}^{-\mez}}\left(1+\lA \nabla_x \sigma\rA_{L^\infty}\right)^\mez
\left(\int f G(\sigma)f\dx\right)^\mez,
$$
where the $\lA \cdot\rA_{\dot{H}^{-1/2}}$-norm is as defined by~\e{defi:homogeneous}.
\end{lemm}
\begin{proof}
Since ${\widehat{\sigma}}(0)=0$, it immediately follows from the Cauchy-Schwarz inequality 
and the Plancherel formula that 
$$
\la \int \sigma f\dx \ra\le (2\pi)^{-d}\lA \sigma\rA_{\dot{H}^{-\mez}}\lA f\rA_{\dot{H}^\mez}
$$
where
\be\label{defi:homogeneous}
\lA \sigma\rA_{\dot{H}^{-\mez}}^2\defn 
\sum_{\xi\in(\xZ^d\setminus\{0\})}\la \xi\ra^{-1}\vert\hat{\sigma}(\xi)\vert^2,\quad
\lA f\rA_{\dot{H}^{\mez}}^2\defn 
\sum_{\xi\in\xZ^d}\la \xi\ra\vert\widehat{f}(\xi)\vert^2.
\ee
Since $h\ge 1$, the wanted inequality immediately 
follows from Proposition~\ref{P:C3}.
\end{proof}

We are now in position to estimate $\int \eta\psi\dx$ in terms of the energy. 

Since $\int_{\xT^d}\eta(t,x)\dx=0$ by assumption, we can apply Lemma~\e{est-psi} to 
get that, for all time $t\in [0,T]$,
$$
\la \int_{\xT^d} \eta(t,x)\psi(t,x)\dx \ra \le 
C \lA \eta(t)\rA_{\dot{H}^{-\mez}}\left(1+\lA \nabla_x \eta(t)\rA_{L^\infty}\right)^\mez\left(\int (\psi G(\eta)\psi)(t,x)\dx\right)^\mez.
$$
On the other hand, as already seen, by Stokes  theorem (see~\e{positivityDN})  we have,
$$
\int_{\xT^d} \psi G(\eta)\psi\dx= \iint_{\Omega}\la \nabla_{x,y}\phi\ra^2\dydx.
$$
Since $g>0$, we have $\iint_{\Omega}\la \nabla_{x,y}\phi\ra^2\dydx \le E$ where $E$ denotes the total energy. Remembering that $E$ is constant, we see that, by multiplying and dividing by $\sqrt{E}$, the previous result implies that, for all $t\in [0,T]$,
$$
\la \int_{\xT^d} \eta(t,x)\psi(t,x)\dx \ra \le 
C Q'E
\quad\text{ where }\quad
Q'=\sup_{t\in [0,T]}\frac{\lA \eta(t)\rA_{\dot{H}^{-\mez}}\left(1+\lA \nabla_x \eta(t)\rA_{L^\infty}\right)}{\sqrt{E}}\cdot
$$
This in turn implies the wanted estimate \e{virialestimate} in light of the obvious inequalities
$$
\frac{\lA \eta(t)\rA_{\dot{H}^{-\mez}}}{\sqrt{E}}
\le \frac{\lA\eta(t)\rA_{\dot{H}^{-\mez}}}{\sqrt{E_p(t)}}\les
\frac{\lA \eta(t)\rA_{\dot{H}^{-\mez}}}{\lA \eta(t)\rA_{L^2}}\le 1.
$$

\section{Higher order virial identities}

In this section we shall prove the identities stated in the introduction in~\S\ref{S:1.4}. Recall that we assume that $g\in\xR$, $h\in (0,+\infty]$ and that $(\eta,\psi)$ 
is a regular solution to the water-wave equations defined on the time interval $[0,T]$. 
We denote by $\phi$ the harmonic extension of $\psi,$  solution to, 
\be\label{defi:phiintro}
\Delta_{x,y}\phi=0 \quad\text{in }\Omega=\{-h<y<\eta(x)\},\quad 
\phi\arrowvert_{y=\eta}=\psi,\quad \partial_y\phi\arrowvert_{y=-h}=0.
\ee
When $h=+\infty$, the last condition means that $\lim_{y\to -\infty}\partial_y\phi=0$, which is encoded in the fact that $\phi$ is a variational solution. 

As explained in~\S\ref{S:1.4}  we introduce, 
$$
B= (\partial_y \phi)\arrowvert_{y=\eta},\quad 
V = (\nabla_x \phi)\arrowvert_{y=\eta},\quad \ma=-\partial_y P\arrowvert_{y=\eta}.
$$

\subsection{Proof of Proposition~\ref{P:VAC1}} 

We begin by proving that the coefficient $\gamma=1-G(\eta)\eta$ is non-negative. 
\begin{lemm}\label{Coro:Zaremba-Taylor}
For any $h\in(0,+\infty]$ and for any Lipschitz function 
$\sigma\in W^{1,\infty}(\xT^d)$, 
there holds,
$$
G(\sigma)\sigma\le 1.
$$
\end{lemm}
\begin{rema}
In addition, if $\sigma\in C^{1,\alpha}(\xT^d)$ 
with $\alpha>0$, then 
$G(\sigma)\sigma<1$. The latter property is proved in 
\cite{AMS,NPausader} and it is related to the fact that the Taylor coefficient is positive (see~\cite{CCFGLF-Annals-2012,CCFG-ARMA-2013,CCFG-ARMA-2016,Cheng-Belinchon-Shkoller-AdvMath}).
\end{rema}
\begin{proof}
We will prove that,
$$
\frac{G(\sigma)\sigma+|\nabla \sigma|^2}{1+|\nabla \sigma|^2}\le 1,
$$
equivalent to the wanted result. It follows from the chain rule that 
$$
\frac{G(\sigma)\sigma+|\nabla \sigma|^2}{1+|\nabla \sigma|^2}= (\partial_y \phi)\arrowvert_{y=\sigma},
$$
where $\phi$ is the variationnal solution to,
\begin{equation*}
\begin{aligned}
\Delta_{x,y}\phi&=0\quad\text{in }\, \Omega = \{(x,y)\,:\, x\in \xT^d,~-h<y<\sigma(x)\},\\
\phi_{\arrowvert y=\sigma(x)}&=\sigma(x),
\end{aligned}
\end{equation*}
where $h \leq +\infty$. Hence the proof reduces to showing that
$(\partial_y \phi)\arrowvert_{y=\sigma}\le 1$.

$i)$ Assume that $h<+\infty.$   
Introduce $p=\phi-y$. 
Then $p$ is an harmonic function in $\Omega$ 
vanishing on $\Sigma\defn \{y=\sigma(x)\}$. 
Moreover, since~$\partial_np\arrowvert_{y=-h}=-\partial_y p\arrowvert_{y=-h}=1>0$, 
one infers from the Hopf-Zaremba principle that $p$ cannot reach its minimum on $\{y=-h\}$.  
So $p$ reaches its minimum on $\Sigma$. This proves that 
$p\ge 0$ in $\Omega$ which implies that 
$\partial_yp\le 0$ on~$\Sigma$. So $\partial_y \phi\le 1$ on $\Sigma$, which completes the proof.

$ii)$ Assume now that $h=+\infty$. Given $\ell>0$ we set,
$$
\Omega_\ell=\left\{ (x,y)\in\xT^d\times\xR\,;\, -\ell<y<\sigma(x)\right\},
$$
and we introduce,
$p=\phi-y$. 
Then $p$ is an harmonic function in $\Omega_\ell$ vanishing on $\Sigma\defn \{y=\sigma(x)\}$. 
Moreover, since~$\partial_y\phi$ goes to $0$ when $y$ goes to $-\infty$ (in any Sobolev norm, see~Lemma~$3.1$ in \cite{AD}), 
one gets that, if~$\ell$ is large enough, then,
$\partial_yp\arrowvert_{y=-\ell}<0$. 
On the other hand, one has $\partial_np=-\partial_y p$ on $\{y=-\ell\}$. Therefore, by repeating the previous arguments, we deduce that $p$ reaches its minimum on $\Sigma$ which in turn implies that $\partial_y \phi\le 1$ on $\Sigma$. The proof 
is complete.
\end{proof}

We then prove~\e{vac:i}. To do this, we will use the Rellich identity already seen earlier (see~Lemma~\ref{intN}) combined with Lannes' shape derivative 
formula for the Dirichlet-to-Neumann operator 
(recalled in appendix, 
see Proposition~\ref{P:shape}).

Write,
\be\label{hve:n9}
\fract\int_{\xT^d}\eta G(\eta)\psi\dx=\int_{\xT^d}(\partial_t\eta) G(\eta)\psi\dx+\int_{\xT^d}\eta (\partial_tG(\eta)\psi)\dx.
\ee
The first step is to calculate $\partial_t G(\eta)\psi$. 
This is where we use the 
shape derivative formula for 
the Dirichlet-to-Neumann operator 
(see Proposition~\ref{P:shape}). 
It implies that
\be\label{hve:n10}
\partial_tG(\eta)\psi=
G(\eta)(\partial_t\psi-B\partial_t\eta)-\cnx(V\partial_t\eta).
\ee
We now assert that
\be\label{hve:n11}
\partial_t\psi-B\partial_t\eta+\mez\big(\la V\ra^2+B^2\big)+g\eta=0.
\ee
This can be verified by an elementary computation using the equations for $\eta$ and $\psi$ (see~\e{systemT}), as well as the expressions~\e{defi:BVbis} for $B$ and $V$. 
Alternatively, 
one can easily obtain a formal proof by starting with the chain rule:
$$
\partial_t\psi-B\partial_t\eta=(\partial_t\phi)\arrowvert_{y=\eta},
$$
and then evaluate the Bernoulli equation,
$$
\partial_t\phi+\mez\vert\nabla_{x,y}\phi\vert^2+P+gy=0
$$
on $y=\eta$, remembering that on the free surface we have $P=0$, $V=\nabla_x\phi\arrowvert_{y=\eta}$ and $B=\partial_y\phi\arrowvert_{y=\eta}$.  

So, using the previous identities 
\e{hve:n10}--\e{hve:n11}, and remembering that $G(\eta)$ is self-adjoint, we get
\begin{align*}
\int_{\xT^d}\eta (\partial_tG(\eta)\psi)\dx
&=-\int_{\xT^d}\eta G(\eta)\Big( \mez\big(\la V\ra^2+B^2\big)+g\eta\Big)\dx-\int_{\xT^d}\eta(\cnx(V\partial_t\eta))\dx\\
&=-\int_{\xT^d}G(\eta)\eta\Big( \mez\big(\la V\ra^2+B^2\big)+g\eta\Big)\dx+\int_{\xT^d}\nabla_x\eta\cdot (V\partial_t\eta))\dx.
\end{align*}
Plugging this identity into \e{hve:n9} and then using that
$$
\partial_t\eta=G(\eta)\psi=B-V\cdot\nabla_x\eta,
$$
we eventually deduce that
\be\label{hve:n12}
\fract\int_{\xT^d}\eta G(\eta)\psi\dx=\int_{\xT^d}B G(\eta)\psi\dx- \int_{\xT^d}G(\eta)\eta\Big( \mez\big(\la V\ra^2+B^2\big)+g\eta\Big)\dx.
\ee
Now observe that
$$
B G(\eta)\psi=B^2-BV\cdot\nabla_x \eta=
\mez (B^2+\la V\ra^2)+\mez (B^2-\la V\ra^2)-BV\cdot\nabla_x \eta.
$$
The Rellich identity (see Lemma~\ref{intN} in the case where $h=+\infty$) then implies that
$$
\int_{\xT^d}\Big(\mez (B^2-\la V\ra^2)-BV\cdot\nabla_x \eta \Big)\dx=-\int_{\xT^d}N\dx=-\mez\int_{\xT^d} \vert \nabla_x \phi(t,x,-h)\vert^2\dx.
$$
Therefore, we have
$$
\int_{\xT^d}B G(\eta)\psi\dx=\mez \int_{\xT^d}(B^2+\la V\ra^2)\dx-\mez\int_{\xT^d} \vert \nabla_x \phi(t,x,-h)\vert^2\dx.
$$
We conclude that
\begin{align*}
\fract\int_{\xT^d}\eta G(\eta)\psi\dx&=\mez \int_{\xT^d}(1-G(\eta)\eta)(B^2+\la V\ra^2)\dx-g \int_{\xT^d}\eta G(\eta)\eta\dx\\
&\quad -\mez\int_{\xT^d} \vert \nabla_x \phi(t,x,-h)\vert^2\dx.
\end{align*}
Finally, using the expressions \e{defi:BVbis}, we verify by a direct computation that $B^2+\la V\ra^2$ simplifies to
$$
B^2+\la V\ra^2=\frac{(G(\eta)\psi)^2+\la\nabla_x\psi\ra^2+\left(\la \nabla_x\eta\ra^2\la\nabla_x\psi\ra^2-(\nabla_x\eta\cdot\nabla_x\psi)^2\right)}{1+\la\nabla_x\eta\ra^2}.
$$
This concludes the proof of~\e{vac:i}.

Then, \e{vac:ii} follows by letting $h$ goes to $+\infty$ using~\e{decay:var2}.
This completes the proof of Proposition~\ref{P:VAC1}.

\subsection{Proof of Corollary~\ref{coro:C.1.9}}

By integrating \e{vac:i} in time  we get,
$$
\frac{1}{T}\int_0^T\int_{\xT^d}\frac{\gamma}{2}\big(B^2+\la V\ra^2\big)(t,x)\dx\dt=\frac{1}{T}\int_0^T(\eta G(\eta)\eta)(t,x)\dx\dt +R(T),
$$
where
$$
R(T)=\frac{1}{2T}\fract\int_{\xT^d}\eta^2\dx\Big\arrowvert_{t=0}^{t=T}.
$$
In addition, for periodic in time solutions, we have $R(T)=0$. 
Since, 
$$
\mez\fract\int_{\xT^d}\eta^2\dx=\int_{\xT^d}\eta G(\eta)\psi\dx,
$$
we see that Corollary~\ref{coro:C.1.9} will be a straightforward consequence of the following result.
\begin{prop}\label{P:4.1}
Let $d\ge 1$ and $h\in (0,+\infty]$. Consider a  
function $\eta\in W^{1,\infty}(\xT^d)$ satisfying,
$$
\int_{\xT^d}\eta(x)\dx=0 \quad\text{and}\quad \inf_{x\in\xT^d}\eta(x)>-\mez h.
$$
(In the infinite depth case $h=+\infty$, the last condition is automatically satisfied). 

$i)$ For all $\psi\in H^{1/2}(\xT^d)$  there holds,  
\be\label{vac:n20}
\la \int_{\xT^d}\eta G(\eta)\psi\dx\ra\le 
\left(\int_{\xT^d}\eta G(\eta)\eta\dx\right)^\mez\left(\int_{\xT^d}\psi G(\eta)\psi\dx\right)^\mez.
\ee

$ii)$ If $h<+\infty$  then, 
\be\label{vac:n21}
0\le \int_{\xT^d}\eta G(\eta)\eta\dx\le h \bla \xT^d\bra.
\ee
$iii)$ If $h=+\infty$  then,
\be\label{vac:n21infty}
0\le \int_{\xT^d}\eta G(\eta)\eta\dx\le \big\lvert \inf_{x\in\xT^d}\eta(x)\big\rvert \cdot \bla \xT^d\bra.
\ee
\end{prop}
\begin{rema}
In the finite depth case, the inequality~\e{vac:n21} is a consequence of \cite[Section $2.3$]{HP22note}.
\end{rema}
\begin{proof}
$i)$ Set $\Omega=\{(x,y)\,;\, x\in\xT^d,~-h<y<\eta(x)\}$. Introduce the harmonic extensions of $\psi$ and $\eta$ which are defined by,
\begin{alignat*}{4}
&\Delta_{x,y}\phi=0 \quad&&\text{in }\Omega,\quad 
&&\phi\arrowvert_{y=\eta}=\psi,\quad &&\partial_y\phi\arrowvert_{y=-h}=0,\\
&\Delta_{x,y}\theta=0 \quad&&\text{in }\Omega,\quad  
&&\theta\arrowvert_{y=\eta}=\eta,\quad  &&\partial_y\theta\arrowvert_{y=-h}=0.
\end{alignat*}
When $h=+\infty$, the 
boundary conditions at $y=-h$ 
mean that $\lim_{y\to -\infty}\partial_y\phi=0$ and 
$\lim_{y\to -\infty}\partial_y\theta=0$, which, as already mentioned, is encoded in the fact that $\phi$ and $\theta$ are variational solutions,
 see~\e{decay:var2}.

It follows from the definition of the Dirichlet-to-Neumann operator and the divergence theorem that
\begin{align*}
\int_{\xT^d}\eta G(\eta)\psi\dx
&=\int_{\partial\Omega}\theta \partial_n \phi\dsigma 
 =\iint_{\Omega}\cn_{x,y}\left(\theta \nabla_{x,y}\phi\right)\dydx\\
&=\iint_{\Omega}\nabla_{x,y}\theta \cdot\nabla_{x,y}\phi\dydx.
\end{align*}
So, the Cauchy-Schwarz inequality implies that
$$
\la \int_{\xT^d}\eta G(\eta)\psi\dx\ra\le \left(\iint_{\Omega}\la\nabla_{x,y}\theta\ra^2\dydx\right)^\mez
\left(\iint_{\Omega}\la\nabla_{x,y}\phi\ra^2\dydx\right)^\mez.
$$
Using again the divergence theorem, we conclude that
$$
\la \int_{\xT^d}\eta G(\eta)\psi\dx\ra\le \left(\int_{\xT^d}\eta G(\eta)\eta\dx\right)^\mez
\left(\int_{\xT^d}\psi G(\eta)\psi\dx\right)^\mez,
$$
which is the wanted inequality~\e{vac:n20}. 

$ii)$ To prove~\e{vac:n21}, we use a variant of the computations above, noticing that 
the functions $y$ and $\theta$ are both equal to $\eta$ on the free surface. 
Namely, we write
\begin{align*}
\int_{\xT^d}\eta G(\eta)\eta\dx
&=\int_{\partial\Omega}y\partial_n \theta\dsigma
=\iint_{\Omega}\cn_{x,y}\left(y\nabla_{x,y}\theta\right)\dydx\\
&=\iint_{\Omega} \partial_{y}\theta\dydx.
\end{align*}
Hence, it follows from the Cauchy-Schwarz inequality that,
$$
\la \int_{\xT^d}\eta G(\eta)\eta\dx\ra\le 
\la \Omega\ra^\mez\left(\iint_{\Omega} (\partial_{y}\theta)^2\dydx\right)^\mez.
$$
Since,
$$
\iint_{\Omega} (\partial_{y}\theta)^2\dydx\le 
\iint_{\Omega} \la\nabla_{x,y}\theta\ra^2\dydx=\int_{\xT^d}\eta G(\eta)\eta\dx,
$$
we conclude that the integral $I\defn\int_{\xT^d}\eta G(\eta)\eta\dx$ satisfies $I\le \sqrt{\la\Omega\ra}\sqrt{I}$, equivalent to the desired inequality \e{vac:n21} since $\Omega=h\bla\xT^d\bra$ in light of the assumption $\int_{\xT^d}\eta\dx=0$.

$iii)$ To prove \e{vac:n21infty} we use a different argument. Since $\eta$ is bounded, we may introduce  $C=\vert \inf_{\xT^d}\eta(x)\vert$.   Now, since we are in infinite depth, it is easily verified that, for any functions $\sigma=\sigma(x)$ and $\psi=\psi(x)$ and any constant $K$, there holds $G(\sigma+K)\psi=G(\sigma)\psi$. On the other hand, $G(\sigma)K=0$. As a result,
$$
\int_{\xT^d}\eta G(\eta)\eta\dx=
\int_{\xT^d}(\eta+C) G(\eta+C)(\eta+C)\dx.
$$
Now, we use Lemma~\ref{Coro:Zaremba-Taylor} with $\sigma=\eta+C$ to write $G(\eta+C)(\eta+C) \le 1$. 
Since $\eta+C\ge 0$ by definition of $C$, this implies that
$$
\int_{\xT^d}\eta G(\eta)\eta\dx\le 
\int_{\xT^d}(\eta+C) \dx=C\bla \xT^d\bra,
$$
which completes the proof of \e{vac:n21infty}. 

This completes the proof of Proposition~\ref{P:4.1}.
\end{proof}

Let us prove another surprising consequence of Proposition~\ref{P:4.1}.

\begin{prop}
Let $g\ge 0$, $h=+\infty$ 
and consider a 
regular solution $(\eta,\psi)$ to the water-wave system defined on the time interval $[0,T]$. For any time $t\in (0,T)$  there holds
\be\label{rofc:ke}
\la \fract E_k(t)\ra^2=\la \fract E_p(t)\ra^2\le g^2
\lA \eta(t)\rA_{L^\infty}\la \xT^d\ra \cdot E_k(t).
\ee
\end{prop}
\begin{rema}
To explain why this result is surprising, we will compare it with the corresponding estimate for the linearized water wave system. In infinite depth, the later system 
reads
$$
\partial_t \eta=\la D_x\ra \psi\quad,\quad \partial_t \psi+g\eta=0.
$$
Then
\begin{align*}
\la \fract E_p(t)\ra^2&=\la \fract \frac{g}{2}\int_{\xT^d}\eta^2(t,x)\dx\ra^2\\
&=\la 
g\int_{\xT^d}\eta(t,x)\la D_x\ra\psi(t,x)\dx\ra^2\\
&\le g^2\lA \eta(t)\rA_{\dot{H}^\mez}^2\lA \psi(t)\rA_{\dot{H}^\mez}^2 = g^2 \lA \eta(t)\rA_{H^\mez}^2 E_k(t).
\end{align*}
In particular, for the nonlinear problem, the upper bound of \e{rofc:ke} is linear in $\eta$ while it is quadratic for the linearized equation. 
\end{rema}
\begin{proof}
$i)$ The fact that $\la \fract E_k(t)\ra=\la \fract E_p(t)\ra$ follows from the conservation of energy. For the sake of completeness, we start by recalling 
the proof of this classical result.

It follows from the kinematic boundary condition that, for any regular function $f=f(t,x,y)$, one has
$$
\fract
\iint_{\Omega(t)} f(t,x,y)\dydx
=\iint_{\Omega(t)} (\partial_t +u\cdot \nabla_{x,y})f \dydx.
$$
It follows that
$$
\fract E_k=\iint_{\Omega(t)}\mez(\partial_t+u\cdot\nabla_{x,y})\la u\ra^2\dydx.
$$
Now observe that
\begin{alignat*}{2}
\mez(\partial_t+u\cdot\nabla_{x,y})\la u\ra^2
&=u\cdot(\partial_t u+(u\cdot\nabla_{x,y})u)\qquad &&(\text{Leibniz rule})\\
&=-u\cdot\nabla_{x,y}(P+gy)\qquad &&(\text{by }\e{Euler})\\
&=-\cnx_{x,y}((P+gy)u)\qquad &&(\text{since }\cnx_{x,y}u=0).
\end{alignat*}
Consequently, Stokes' theorem implies that
$$
\fract E_k=-\iint_{\Omega(t)}\cnx_{x,y}((P+gy)u)\dydx=-\int_{\partial\Omega(t)}(P+gy)u\cdot n\dsigma.
$$
Now, by assumption on the pressure, on the free surface $\Sigma(t)$ 
we have $P+gy=g\eta$. 
Also the kinematic boundary condition \e{kinematics} implies that 
$u\cdot n\dsigma=\partial_t \eta \dx$ (since $\dsigma=\sqrt{1+|\nabla_x \eta|^2}\dx$). It follows that
\be\label{conservedEproof}
\fract E_k=-\int_{\xR^2} g\eta  \, \partial_t \eta\dx=-\fract E_p.
\ee
The identity \e{conservedEproof} is proven.

$ii)$ We now prove the upper bound. Now, write
$$
\fract E_p = g\int_{\xT^d}\eta\partial_t \eta\dx=
g\int_{\xT^d}\eta G(\eta)\psi\dx.
$$
Now, recall from \e{vac:n20} that,
\begin{align*}
\la \int_{\xT^d}\eta G(\eta)\psi\dx\ra&\le 
\left(\int_{\xT^d}\eta G(\eta)\eta\dx\right)^\mez\left(\int_{\xT^d}\psi G(\eta)\psi\dx\right)^\mez,\\
&\leq \left(\int_{\xT^d}\eta G(\eta)\eta\dx\right)^\mez E_k^\mez.
\end{align*}
On the other hand by 
\e{vac:n21infty}, we have
$$
0\le \int_{\xT^d}\eta G(\eta)\eta\dx\le \bla \inf_{x\in\xT^d}\eta \bra \cdot \bla \xT^d\bra.
$$
Therefore, by combining the previous inequalities, we see that
$$
\la\fract E_p\ra\le g \sqrt{\bla \inf_{x\in\xT^d}\eta \bra \cdot \bla \xT^d\bra\cdot E_k},
$$
which immediately implies the wanted result.
\end{proof}

\subsection{The Longuet-Higgins formula and Corollary~\ref{C:VAC2}}\label{S:4.2}

For the sake of completeness, we begin by proving the identity~\e{LHintro}, whose statement is recalled here.

\begin{prop}[\cite{LH1974}]
Let $g\in \xR$ and $d\ge 1$. Consider a 
regular solution $(\eta,\psi)$ to the water-wave system defined on the time interval $[0,T]$. Assume that $h<+\infty$. For any time $t\in (0,T)$, there holds
\be\label{LH}
\mez \fractt\int_{\xT^d}\eta^2(t,x)\dx=\int_{\xT^d}(P(t,x,-h)-gh)\dx,
\ee
where $P$ is the pressure as given by \e{t5} where $\phi$ is the  harmonic extension of $\psi$.
\end{prop}
\begin{proof}
Directly from the equation for $\eta$ and using the definition of the 
Dirichlet-to-Neumann operator, 
we get
$$
\mez \fractt\int_{\xT^d}\eta^2(t,x)\dx
=\fract\int_{\xT^d}\eta \partial_t\eta\dx
=\fract \int_{\xT^d}\eta G(\eta)\psi\dx
=\fract\int_{\partial\Omega(t)}y \partial_n\phi\dsigma,
$$
where 
$\dsigma=\sqrt{1+|\nabla_x\eta|^2}\dx$ 
is the surface measure. Then, by using the divergence 
theorem and the fact that $\phi$ is an harmonic function  
we obtain
\begin{align*}
\mez \fractt\int_{\xT^d}\eta^2\dx=
\fract\iint_{\Omega(t)}\cnx_{x,y}(y \nabla_{x,y} \phi)\dydx=\fract \iint_{\Omega(t)}\partial_y\phi\dydx.
\end{align*}
Now, we recall that, for any function $f=f(t,x,y)$  there holds
$$
\fract \iint_{\Omega(t)}f(t,x,y)\dydx=\iint_{\Omega(t)}(\partial_t +\nabla_{x,y}\phi\cdot\nabla_{x,y})f\dydx.
$$
Set $v_y=\partial_y\phi$, that is the vertical component of the eulerian velocity field. 
Since $v_y$ solves
$$
\partial_t v_y+v\cdot\nabla_{x,y}v_y+\partial_yP+g=0,
$$
we deduce that
$$
\mez \fractt\int_{\xT^d}\eta^2(t,x)\dx=-
\iint_{\Omega(t)}\partial_y P(t,x,y)\dydx-g\iint_{\Omega(t)}\dydx.
$$
Since $P$ vanishes on the free surface  we see that
$$
\iint_{\Omega(t)}\partial_y P\dydx=\int_{\xT^d}\left(\int_{-h}^\eta\partial_y P\dy
\right)\dx=-\int_{\xT^d}P(t,x,-h)\dx.
$$
On the other hand, since the volume $\la \Omega(t)\ra$ is independent of time (see Proposition~\ref{eta1}), 
we get
$$
\iint_{\Omega(t)}\dydx=\la\Omega(0)\ra =h\vert \xT^d\vert,
$$
where we have used  $\int_{\xT^d}\eta(0,x)\dx=0$ to get the last equality. 

By combining the previous identities, we obtain the wanted result~\e{LH}.
\end{proof}

We are now in position to prove Corollary~\ref{C:VAC2}. 

By combining~\e{LH} with \e{vac:i}  we see that
\be\label{vac:iii}
\begin{aligned}
\int_{\xT^d}(P(t,x,-h)-gh)\dx&=
\int_{\xT^d}\frac{\gamma}{2}\big(
(G(\eta)\psi(t,x))^2+\la\nabla_x\psi(t,x)\ra^2\big)\dx\\
&\quad -g\int_{\xT^d}(\eta G(\eta)\eta) (t,x)\dx
\\
&\quad -\mez\int_{\xT^d} \vert \nabla_x \phi(t,x,-h)\vert^2\dx.
\end{aligned}
\ee
Remembering that the pressure is given by the Bernoulli  equation, that is
$$
\partial_t \phi+\mez \la \nabla_{x,y}\phi\ra^2+P+gy=0,
$$
and noticing that $\la \nabla_{x,y}\phi\ra^2(t,x,-h)=\la \nabla_{x}\phi\ra^2(t,x,-h)$ (in light of the boundary 
condition $\partial_y\phi=0$ on $y=-h$)  we end up with
\be\label{vac:iv}
\begin{aligned}
-\fract\int_{\xT^d}\phi(t,x,-h)\dx&=
\int_{\xT^d}\frac{\gamma}{2}\big(
(G(\eta)\psi(t,x))^2+\la\nabla_x\psi(t,x)\ra^2\big)\dx\\
&\quad -g\int_{\xT^d}(\eta G(\eta)\eta) (t,x)\dx,
\end{aligned}
\ee
which proves Corollary~\ref{C:VAC2}.

\subsection{Proof of Proposition~\ref{P:VAC3}}

The proof of Proposition~\ref{P:VAC3} relies on the following
\begin{prop}[from \cite{ABZ3}]\label{prop:newS} 
For an arbitrary fluid domain with finite or 
infinite depth ($h\leq +\infty$)  there holds,
\begin{align}
(\partial_{t}+V\cdot\nabla_x)\B&=\ma-g,\label{eq:B}\\
(\partial_t+V\cdot\nabla_x)V+\ma\zeta&=0,\label{eq:V}\\
(\partial_{t}+V\cdot\nabla_x)\zeta&=G(\eta)V+ \zeta G(\eta)\B ,\label{eq:zeta}
\end{align}
where $\zeta = \nabla_x \eta.$
\end{prop}
\begin{proof}
We refer to Proposition~$4.3$ in \cite{ABZ3} for the proof. Notice that in \cite{ABZ3}, the above identities are established for a general fluid domain. Therefore in Proposition~$4.3$ of \cite{ABZ3}, the last equation is valid up to a reminder term denoted here by $\gamma$. However, for a domain with a flat bottom, 
this reminder term $\gamma$ disappears, as we can easily verify (this follows from the fact that, if $\partial_y\phi\arrowvert_{y=-h}=0$ then $(\partial_y\partial_{x_j}\phi)\arrowvert_{y=-h}=0$).
\end{proof}

The proof of \e{virial:ordre1B} is then straightforward. We integrate \e{eq:B},   then we integrate by parts and we use  \e{n991} (using the assumption $h=+\infty$) to obtain,
\begin{align*}
\fract\int_{\xT^d}B\dx
&=-\int_{\xT^d}V\cdot\nabla_x B\dx +\int_{\xT^d}(a-g)\dx,\\
&=\int_{\xT^d}(\cnx V)B\dx +\int_{\xT^d}(a-g)\dx,\\
&=\int_{\xT^d}B G(\eta)B\dx +\int_{\xT^d}(a-g)\dx.
\end{align*}

To prove \e{virial:ordre1V}  we start by 
combining the identities \eqref{n991} and \e{eq:zeta}  to deduce,
$$
\partial_t \zeta+V\cdot \nabla_x\zeta=G(\eta)V-\zeta\cnx(V), \quad \zeta = \nabla_x \eta.
$$
Now taking the scalar product of this last equation with 
$V$ we get,
$$
V\cdot\partial_t\zeta+V\cdot(V\cdot\nabla_x \zeta)
=V\cdot G(\eta)V-\zeta \cdot V(\cnx V).
$$

On the other hand, taking the scalar product of the equation~\e{eq:V} with $\zeta$  we have,
$$
\zeta\cdot \partial_t V+\zeta\cdot (V\cdot\nabla_x V)+a \la \zeta\ra^2=0.
$$
Therefore,
$$
\fract\int_{\xT^d}\zeta\cdot V\dx
=\int_{\xT^d}V\cdot G(\eta)V-\int_{\xT^d}a \la \zeta\ra^2\dx+R
$$
where,
$$
R(t)=-\int_{\xT^d}r(t,x)\dx\quad\text{with}\quad
r=V\cdot(V\cdot\nabla_x \zeta)+\zeta \cdot V(\cnx V)+\zeta\cdot (V\cdot\nabla_x V).
$$
Now observe that $r$ can be written as the divergence of a vector field. Indeed, 
$$
r=\sum_{i,j}V_iV_j\partial_j\zeta_i
+\zeta_iV_i (\partial_jV_j)+\zeta_iV_j(\partial_jV_i)
=\sum_j\partial_j\left(\sum_i \zeta_iV_j\right).
$$
This implies that $R=0$, which completes the proof. 

\section{Rayleigh-Taylor instability}\label{S:4}
We prove here Proposition~\ref{P:RT0} which we recall.

\begin{prop}\label{P:RT0p}
Assume that $g=0$ and fix $h>0$. 
There exists a constant $c>0$ such that, for all regular 
solution $(\eta,\psi)$ to the water-wave system defined 
on the time interval $[0,T]$    there holds,
\be\label{RT:0p}
\lA \eta(t)\rA_{L^2}(1+\lA \nabla_x\eta(t)\rA_{L^\infty})^\mez\ge 
\frac{c}{\sqrt{E}}\left(E t+\int_{\xT^d}\eta(0,x)\psi(0,x)\dx\right),
\ee
where 
$E=\mez \iint_{\Omega(0)}\la \nabla_{x,y}\phi(0,x,y)\ra^2\dydx$.
\end{prop}
 \begin{proof}
If $g=0$, then the potential energy vanishes and the total energy $E$ (which is constant) 
is equal to the kinetic energy. In particular, for all time $t\in [0,T]$, 
$$
\mez \iint_{\Omega(t)}\la \nabla_{x,y}\phi(t,x,y)\ra^2\dydx=\mez \iint_{\Omega(0)}\la \nabla_{x,y}\phi(0,x,y)\ra^2\dydx=E.
$$
Also, for $g=0$, the virial identity~\e{MI1} simplifies to
\begin{align*}
\mez \fract
\int_{\xT^d} \eta(t,x)\psi(t,x)\dx &= \iint_{\Omega(t)} \left(\frac{3}{4}(\partial_y \phi)^2  + \frac{1}{4} \vert \nabla_x \phi\vert^2\right)(t,x,y)\dydx\\ &\quad +\frac{h}{4}\int_{\xT^d} \vert \nabla_x \phi(t, x, -h)\vert^2\dx.
\end{align*}
It follows that,
$$
\mez \fract
\int_{\xT^d} \eta(t,x)\psi(t,x)\dx \ge \frac{1}{4}\iint_{\Omega(t)} \la\nabla_{x,y}\phi\ra^2(t,x,y)\dydx=\mez E.
$$
Since $E$ does not depend on time, by integrating in time, we obtain,  
$$
\int_{\xT^d} \eta(t,x)\psi(t,x)\dx \ge Et +\int_{\xT^d} \eta(0,x)\psi(0,x)\dx.
$$
Now, to estimate the left hand-side, we apply Lemma~\ref{est-psi}  to obtain,
$$
\la \int_{\xT^d} \eta(t,x)\psi(t,x)\dx \ra \le 
C \lA \eta(t)\rA_{\dot{H}^{-\mez}}\left(1+\lA \nabla_x \eta(t)\rA_{L^\infty}\right)^\mez\sqrt{E},
$$
for some constant $C$ depending only on the dimension. Now since $\int_{\xT^d} \eta(t,x)\, dx = 0$ we have $\lA \eta(t)\rA_{L^2(\xT^d)}\geq \lA \eta(t)\rA_{\dot{H}^{-\mez}(\xT^d)},$ so the above inequalities  imply~\e{RT:0p}.
\end{proof}

\begin{coro}
If moreover $\eta(0,\cdot)=0$  then
$$
\lA \eta(t)\rA_{L^2}(1+\lA \nabla_x\eta(t)\rA_{L^\infty})^\mez\ge 
C\lA \psi(0,\cdot)\rA_{\dot{H}^\mez}t.
$$
\end{coro}
\begin{proof}
Indeed, Stokes' theorem implies that
\begin{align*}
 2E&= \iint_{\Omega(0)}\la\nabla_{x,y}\phi(0,x,y)\ra^2\dydx= \int_{\partial\Omega(0)}\phi(0,\cdot) \partial_n \phi(0,\cdot)\dsigma\\
&=  \int_{\xT^d} \psi(0,x) G(0)\psi(0,x)\dx=\int_{\xT^d} \psi(0,x) \vert D_x \vert \text{th}(h \vert D_x \vert)\psi(0,x)\dx\\
&\geq C\lA \psi(0,\cdot)\rA^2_{\dot{H}^\mez}.
\end{align*}
Then we use  \eqref{RT:0p} to conclude.
\end{proof}
We now extend the previous analysis to the case $g< 0$.
\begin{prop}
Assume that $g< 0$ and fix $h\in (0,+\infty]$. 
There exists a constant $C>0$ such that, for all regular solution $(\eta,\psi)$ to the water-wave system defined on the time interval $[0,T]$, there holds
$$
C\lA \eta(t)\rA_{L^{2}}\left(1+\lA \nabla_x \eta(t)\rA_{L^\infty}\right)^\mez\left(E+ \frac{|g|}{2}\lA \eta(t)\rA_{L^{2}}^2\right)^\mez\ge |E| t+ \int_{\xT^d} \eta(0,x)\psi(0,x)\dx,
$$
where $E\in \xR$ is the total energy.
\end{prop}
\begin{rema}
The total energy is constant in time, but not necessarily 
non-negative when $g<0$. However the quantity 
$E+ \frac{|g|}{2}\lA \eta(t)\rA_{L^{2}}^2$ is non-negative in light of 
\e{RT:4} below, so that the square root is well-defined even for $E<0$.
\end{rema}
\begin{proof}[Proof of Proposition \ref{P:RT0}]
Since $g=-\la g\ra$   by assumption, the virial identity~\e{MI1} implies that,
\begin{align}
\mez \fract
\int_{\xT^d} \eta(t,x)\psi(t,x)\dx 
&\ge \frac{1}{4}
\iint_{\Omega(t)} \la \nabla_{x,y} \phi\ra^2(t,x,y)\dydx+
\frac{\la g\ra}{2}\int_{\xT^d}\eta^2(t,x)\dx.
\end{align}
On the other hand, 
for $g<0$  the conservation of the total energy reads,
\be\label{RT:4}
\mez\iint_{\Omega(t)} \la \nabla_{x,y} \phi\ra^2(t,x,y)\dydx
=E+\frac{\la g\ra }{2}\int_{\xT^d}\eta^2(t,x)\dx.
\ee
It follows from \e{RT:4} that 
\begin{alignat*}{3}
&\text{If }E<0 \quad &&\text{then}\quad &&
\frac{| g|}{2}\int_{\xT^d}\eta^2(t,x)\dx\ge  -E=|E|,\\
&\text{If }E\ge 0 \quad &&\text{then}\quad &&
\mez\iint_{\Omega(t)} \la \nabla_{x,y} \phi\ra^2(t,x,y)\dydx\ge  E=|E|.
\end{alignat*}
Consequently,
$$
\mez \fract
\int_{\xT^d} \eta(t,x)\psi(t,x)\dx 
\ge \mez |E|.
$$

Integrating in time  this gives,
\begin{align}
\int_{\xT^d} \eta(t,x)\psi(t,x)\dx 
&\ge  |E| t  + \int_{\xT^d} \eta(0,x)\psi(0,x)\dx .
\end{align}
Remembering that,
$$
\int_{\xT^d} (\psi G(\eta)\psi)(t,x)\dx
=\iint \la\nabla_{x,y}\phi(t,x,y)\ra^2\dydx
=2E+\la g\ra\int_{\xT^d}\eta^2(t,x)\dx,
$$
we conclude the proof by using Lemma~\ref{est-psi}.
\end{proof}

\section{The example of standing waves}\label{S:5}

As an illustration, we are going to compute the mean potential energy and the mean 
modified kinetic energy for a fundamental example of water-waves. 
Namely we are going to consider standing gravity waves (in infinite depth). 
Boussinesq was the first to consider the question of the existence of  such waves, 
already in 1877. The question of the existence of these waves 
has remained open for a very long time. 
The first rigorous existence result of {\em exact} solutions 
is due to Plotnikov and Toland~\cite{PlTo}. This work was extended by 
Iooss, Plotnikov and Toland~(\cite{IPT,IP-SW1}). 

The study of the equipartition of energy for standing waves was already considered by 
Mack and Jay~\cite{MackJay}. 
We will follow in this section their notations 
and express the water-waves equations in convenient 
dimensionless variables. The solutions are of the form 
$\phi= \phi(t,x,y ,\eps)$, $\eta= \eta(t,x,\eps)$ and the angular frequency is 
$\omega = \omega(\eps)$ 
where $ \eps >0$ measures the amplitude of the waves, which are $2\pi$-periodic in time (so that $t \in (0, 2\pi)$). 
The equations read
\begin{equation}\label{KT}
\begin{aligned}
&(i) \quad  \Delta_{x,y} \phi = 0, \quad \text{ in } \{0\leq x \leq \pi, \,  y \leq \eps \eta(t,x, \eps)\}, \quad X = (x,y),\\
&(ii)\quad \eta + \omega \partial_t \phi +  \mez \eps\big((\partial_x\phi) ^2 + (\partial_y\phi) ^2) = 0, \quad \text{on } y = \eps \eta(t,x, \eps),\\
& (iii)\quad \omega \partial_t \eta -\partial_y \phi   + \eps \partial_x \phi \partial_x\eta = 0 \quad  \text{on } y = \eps \eta(t,x, \eps),\\
&(iv)\quad  \int_0^\pi \eta(t,x)\dx = 0,\\
& (v)\quad \nabla_{x,y} \phi (t+ 2 \pi, x,y, \eps) =   \nabla_{x,y} \phi (t, x,y, \eps). 
\end{aligned}
 \end{equation}
(Standing waves are $2\pi$-periodic in $x$ 
and symmetric ($x\mapsto \eta(t,x)$ is even) so that one can reduce the analysis to 
$x\in [0,\pi]$.) 
 
Mack, Jay et Sattler consider approximate solutions of the form,
\begin{align*}
&  \phi(t,x,y,\eps)=    \phi^{(0)}(t,x,y) + \eps   \phi^{(1)}(t,x,y) +   \eps^2  \phi^{(2)}(t,x,y) + \mathcal{O}(\eps^3),\\
& \eta(t,x,\eps) = \eta^{(0)}(t,x) + \eps \eta^{(1)}(t,x) +  \eps^2\eta^{(2)}(t,x)+ \mathcal{O}(\eps^3),\\
&\omega(\eps) = \omega_0 + \eps \omega_1 +  \eps^2 \omega_2+\mathcal{O}(\eps^3),
\end{align*}
where,
\begin{align} 
&\phi^{(0)} =-   \sin t \cos x \, e^y,\quad
\eta^{(0)} = \cos t \cos x, \quad \omega_0 = 1,\label{sol0}\\
&\phi^{(1)}  = \alpha(t),    \quad \eta^{(1)}  = \frac{1}{2}   (\cos  t )^2 \cos (2x),  \quad\omega_1= 0.\label{sol1}
\end{align}
and
\begin{equation}\label{sol2}
\begin{aligned}
\phi^{(2)} &=    A^{(2)}_{13} \sin t \, \cos(3x)\, e^{3y} + \frac{5}{32}\sin (3t)\cos x \,e^y  
 + A^{(2)}_{33} \sin(3t) \cos(3x) e^{3y},\\
 \eta^{(2)} & = \frac{3}{32} \cos t \cos x + b_{13} \cos t\cos (3x)  -\frac{1}{16} \cos(3t) \cos x + b_{33} \cos(3t)\cos(3x),\\
  \omega_2 &= -\frac{1}{8}, 
\end{aligned}
 \end{equation}
 where $ \alpha(t)$ is a function depending only on $t$ and $A^{(2)}_{13}, A^{(2)}_{33}, b_{13},b_{33}$ are real  numbers.
 
 We shall verify on this solution the equality in    Corollary \ref{C:virial}    in the case of infinite depth, modulo $\eps^3$ that is,  
 \begin{equation}\label{rep-energ}
 \int_0^{2\pi} \int_0^\pi \int_{-\infty}^{\eps \eta} \big\{\frac{3}{2} \big(\phi_y\big)^2 + \frac{1}{2} \big(\phi_x\big)^2 \big\}(t,x,y)\dy\dx\dt = 2 \int_0^{2\pi} E_p(t)\dt + \mathcal{O}(\eps^3).
 \end{equation}
\subsection*{Computation of the modified kinetic  energy.}
 Setting $f_x = \partial_x f, f_y = \partial_y f,$ we have,  
\begin{equation}\label{derphi}
\begin{aligned}
\phi^{(0)}_x &=   \sin t \sin x \, e^y,\quad 
\phi^{(0)}_y = -  \sin t \cos x \, e^y, \\
\phi^{(1)}_x &= 0, \quad
\phi^{(1)}_y = 0  ,\\
\phi^{(2)}_x&=  -3A^{(2)}_{13} \sin t \, \sin (3x)\,e^{3y}  - \frac{5}{32}\sin (3t)\sin  x \, e^{y} 
  -3 A^{(2)}_{33} \sin(3t) \sin(3x) e^{3y},\\ 
   \phi^{(2)}_y &= 3A^{(2)}_{13} \sin t \, \cos(3x)\,e^{3y}  +  \frac{5}{32}\sin (3t)\cos x \, e^{y} 
  + 3A^{(2)}_{33} \sin(3t) \cos(3x)e^{3y}.\\
   \end{aligned}
 \end{equation}
Then,
\begin{align*}
 (\phi_x)^2 &= (\phi^{0}_x)^2  +2 \eps^2 \big(    \phi^{(0)}_x  \, \phi^{(2)}_x \big) + \mathcal{O}(\eps^3),\\
 (\phi_y)^2 &= (\phi^{0}_y)^2    + 2\eps^2 \big(   \phi^{(0)}_y  \, \phi^{(2)}_y \big) + \mathcal{O}(\eps^3).
 \end{align*}
Set,
 \begin{equation}\label{phi-x}
 \begin{aligned}
 A_1 &= \int_0^\pi \int_{-\infty}^{\eps \eta} (\phi^{0}_x)^2(t,x,y)\dy\dx\\
 A_2 &=  2\eps^2 \int_0^\pi \int_{-\infty}^{\eps \eta}\big(    \phi^{(0)}_x  \, \phi^{(2)}_x \big)(t,x,y)\dy\dx
 \end{aligned}
 \end{equation}
 We have, 
$$ A_1 = \int_0^\pi(\sin x)^2\Big(\int_{-\infty}^{\eps \eta} e^{2y}\dy\Big)\dx (\sin t)^2= \mez \Big(\int_0^\pi(\sin x)^2 e^{2\eps \eta}dx\Big) (\sin t)^2.$$
 We have, modulo $\mathcal{O}(\eps^3)$,
 \begin{equation}\label{e^eta}
 \begin{aligned}
   e^{2\eps \eta} &\equiv 1+ 2\eps \eta + 2 \eps^2 \eta^2  \equiv 1+ 2\eps \eta^{(0)} + 2\eps^2( (\eta^{(0)})^2 +  \eta^{(1)}),\\
   &\equiv 1+ 2 \eps  \cos t \cos x + 2 \eps^2\big\{(\cos t)^2 ( \cos x)^2 + \frac{1}{2}(\cos t)^2\cos(2x)\big\}.
   \end{aligned}  
   \end{equation} 
  Since,
   \begin{align*}
    &\int_0^\pi (\sin x)^2  \dx = \frac{\pi}{2}, \quad \int_0^\pi (\sin x)^2 \cos x \dx = 0, \quad \int_0^\pi (\sin x)^2 (\cos x)^2 \dx = \frac{\pi}{8},\\
    &\int_0^\pi (\sin x)^2 \cos (2x)\dx  = - \frac{\pi}{4}
   \end{align*} we obtain,
   $$A_1 =\frac{\pi}{4}(\sin t)^2 +   \eps^2\big\{ \frac{\pi}{8} (\cos t)^2(\sin t)^2 - \frac{\pi}{8}(\cos t)^2(\sin t)^2  \big\} = \frac{\pi}{4}(\sin t)^2.$$
    \begin{equation}\label{A1}
    A_1 =\frac{\pi}{4}(\sin t)^2    + \mathcal{O}(\eps^3). 
    \end{equation}
    In   computing   $A_2$ we shall write $e^{n\eps \eta } = 1 + \mathcal{O}(\eps)$. Since, 
     $\int_0^\pi \sin x \sin (3x)\dx = 0,$ 
    $$A_2 =  \eps^2 \Big(\int_0^\pi - \frac{5}{32} (\sin x)^2\dx\Big) \sin t \sin (3t) + \mathcal{O}(\eps^3)= - \frac{5\pi}{64}\eps^2\sin t \sin (3t) + \mathcal{O}(\eps^3).$$
    Since $\sin t \sin (3t) = 3 (\sin t)^2 - 4 (\sin t)^4$ we have,
    \begin{equation}\label{A2}
    A_2 = -\eps^2\big\{ \frac{15\pi}{64} (\sin t)^2 - \frac{5\pi}{16} (\sin t)^4\big\}+ \mathcal{O}(\eps^3).
    \end{equation} 
    It follows from \eqref{A1},   \eqref{A2} that,
    \begin{equation}\label{phix}
     \int_0^\pi \int_{-\infty}^{\eps \eta} (\phi_x)^2(t,x,y)\dy\dx = \frac{\pi}{4}(\sin t)^2  -\eps^2\big\{ \frac{15\pi}{64} (\sin t)^2 - \frac{5\pi}{16} (\sin t)^4\big\}+ \mathcal{O}(\eps^3).
    \end{equation}

    Set,
 \begin{equation*} 
 \begin{aligned}
 B_1 &= \int_0^\pi \int_{-\infty}^{\eps \eta} (\phi^{0}_y)^2(t,x,y)\dy\dx,\\
 B_2 &=  2\eps^2 \int_0^\pi \int_{-\infty}^{\eps \eta}\big(    \phi^{(0)}_y  \, \phi^{(2)}_y \big)(t,x,y)\dy\dx.
 \end{aligned}
 \end{equation*}
 We have, 
 $$B_1 =  \int_0^\pi(\cos x)^2\Big(\int_{-\infty}^{\eps \eta} e^{2y}\dy\Big)\dx (\sin t)^2= \mez \Big(\int_0^\pi(\cos x)^2 e^{2\eps \eta}dx\Big) (\sin t)^2.$$
 Using \eqref{e^eta} and the fact that,
 \begin{align*}
  &\int_0^\pi(\cos x)^2\dx = \frac{\pi}{2}, \quad \int_0^\pi(\cos x)^3\dx =0,  \quad \int_0^\pi(\cos x)^4\dx = \frac{3\pi}{8},\\ &\int_0^\pi(\cos x)^2\cos(2x)\dx =\frac{\pi}{4},
 \end{align*}
 we obtain, 
 $$B_1 = \frac{\pi}{4} (\sin t)^2 + \eps^2 \big\{ \frac{3\pi}{8}(\sin t)^2(\cos t)^2 +\frac{\pi}{8}(\sin t)^2(\cos t)^2\big\} + \mathcal{O}(\eps^3),$$
 \begin{equation}\label{B1}
 B_1 =  \frac{\pi}{4} (\sin t)^2  + \frac{\pi}{2} \eps^2( (\sin t)^2- (\sin t)^4)+\mathcal{O}(\eps^3). \end{equation}
 In   computing   $B_2$ we   write $e^{n\eps \eta } = 1 + \mathcal{O}(\eps)$. Since, 
     $\int_0^\pi \cos x \cos(3x)\dx = 0$  we get,
     $$B_2 = -\frac{5}{32} \int_0^\pi(\cos x)^2\dx (\sin t)(\sin 3t) + \mathcal{O}(\eps^3)= - \frac{5\pi}{64}\eps^2\sin t \sin (3t) + \mathcal{O}(\eps^3),$$
so,
\begin{equation}\label{B2}
B_2 = -\eps^2\big\{ \frac{15\pi}{64} (\sin t)^2 - \frac{5\pi}{16} (\sin t)^4\big\}+ \mathcal{O}(\eps^3).
\end{equation}
It follows from\eqref{B1},  \eqref{B2} that,
\begin{equation}\label{phiy}
 \int_0^\pi \int_{-\infty}^{\eps \eta} (\phi_y)^2(t,x,y)\dy\dx = \frac{\pi}{4}(\sin t)^2 + \eps^2\big\{ \frac{17\pi}{64}(\sin t)^2 - \frac{3 \pi}{16}(\sin t)^2\big\}.
\end{equation}
We deduce from  \eqref{phix} and \eqref{phiy} that,
\begin{equation}\label{phiy+phix}
  \int_0^\pi \int_{-\infty}^{\eps \eta} \Big\{ \frac{3}{2}(\phi_y)^2 + \frac{1}{2}(\phi_x)^2\Big\}(t,x,y)\dy\dx=\frac{\pi}{2}(\sin t)^2 + \pi \eps^2 \Big\{ \frac{9}{32}(\sin t)^2 - \frac{1}{8} (\sin t)^4\Big\}.
  \end{equation}
Since, $\int_0^{2\pi} (\sin t)^2\dt = \pi $ et  $\int_0^{2\pi} (\sin t)^4\dt = \frac{3\pi}{4}$ we get,
\begin{equation}\label{EM}
 \int_0^{2\pi}\int_0^\pi \int_{-\infty}^{\eps \eta} \Big\{ \frac{3}{2}(\phi_y)^2 + \frac{1}{2}(\phi_x)^2\Big\}(t,x,y)\dy\dx = \frac{\pi^2}{2} + \frac{3\pi^2}{16} \eps^2 + \mathcal{O}(\eps^3).
\end{equation}

\subsection*{Computation of the potential  energy }
Recall that,
\begin{align*}
 2E_p(t) &=   \int_0^\pi (\eta(t,x))^2\dx = P_0(t) + \eps P_1(t) + \eps^2 P_2(t) +  \mathcal{O}(\eps^3),\\
 P_0(t) &=    \int_0^\pi  (\eta^{(0)}(t,x))^2 \dx,\\
 P_1(t) &=    2\int_0^\pi (\eta^{(0)} \eta^{(1)})(t,x)\dx,\\
 P_2(t)&=    \int_0^\pi \big\{(\eta^{(1)})^2 +  2\eta^{(0)} \eta^{(2)}\big\} (t,x) \dx.
  \end{align*}
 From  \eqref{sol0}  we have,
 \begin{equation}\label{P0}
  P_0(t) = \Big(\int_0^\pi (\cos x)^2\dx \Big)(\cos t)^2 = \frac{\pi}{2}(\cos t)^2.
  \end{equation}
Now,
 \begin{equation}\label{P1}
 P_1(t) =\frac{1}{2} (1+  \cos (2t))\cos t  \Big(\int_0^\pi \cos x \cos(2x)\dx\Big)= 0.
 \end{equation} 
On the other hand, since, $\int_0^\pi\cos(2x))^2\dx = \frac{\pi}{2}$  we have, 
$$
  \int_0^\pi  (\eta^{(1)}(t,x))^2 \dx  = \frac{\pi}{32}\big\{1+  \cos(2t)\big\}^2, 
   = \frac{\pi}{32}\big\{1 - (\sin t)^2\big\}^2  
$$
 Now, since, $\int_0^\pi\cos x \cos(3x)\dx = 0$ and $ \int_0^\pi(\cos x)^2\dx = \frac{\pi}{2}$ we get,
 $$  2\int_0^\pi (\eta^{(0)} \eta^{(2)})(t,x)   \dx =  \pi \big\{ \frac{3}{32} (\cos t)^2   - \frac{1}{16} \cos t \cos(3t) \big\}.$$
  \begin{equation}\label{P2}
 \begin{aligned}
 P_2(t) =  \frac{\pi}{32}\big\{1-  (\sin t)^2\big\}^2 +  \pi \big\{ \frac{3}{32} (\cos t)^2   - \frac{1}{16} \cos t \cos(3t) \big\}. 
 \end{aligned}
  \end{equation}
  It follows from \eqref{P0}, \eqref{P1}, \eqref{P2} that,
 \begin{equation}\label{Ep-infini}
 2E_p(t)=  \frac{\pi}{2} - \frac{\pi}{2}(\sin t)^2 + \eps^2\Big\{ \frac{5\pi}{32} - \frac{\pi}{32}(\sin t)^2 - \frac{\pi}{8}( \sin t )^4 \Big\} +  \mathcal{O}(\eps^3).
 \end{equation}
Therefore,
\begin{equation}\label{Ep-M}
 2 \int_0^{2\pi} E_p(t)\dt = \frac{\pi^2}{2} + \frac{3 \pi^2}{16}\eps^2+ \mathcal{O}(\eps^3).
 \end{equation}
 It follows from \eqref{EM} and \eqref{Ep-M} that the equality \eqref{rep-energ} is satisfied modulo $\mathcal{O}(\eps^3).$

 \appendix
 
\section{The Craig-Sulem-Zakharov formulation}\label{Appendix:DN}

For the reader convenience, we recall in this appendix various well-known results and identities.

\subsection{The Dirichlet-to-Neumann operator}
In this paragraph the time variable is seen as a parameter and we skip it. 

Consider a smooth function $\eta\in C^{\infty}(\xT^d)$ 
and a function $\psi$ in the Sobolev 
space $H^{\mez}(\xT^d)$. 
Then it follows from classical arguments that there is a 
unique variational solution 
$\phi$ to the problem
\be\label{defi:phi}
\Delta_{x,y}\phi=0 \quad\text{in }\Omega=\{-h<y<\eta(x)\},\quad 
\phi\arrowvert_{y=\eta}=\psi,\quad \partial_y\phi\arrowvert_{y=-h}=0.
\ee
By construction, the variational solution is such that $\nabla_{x,y}\phi\in L^2(\Omega)$, 
so it is not obvious that one can consider the trace 
$\partial_n \phi\arrowvert_{\partial\Omega}$. However, 
since $\Delta_{x,y}\phi=0$, one can express the normal 
derivative in terms of the tangential derivatives 
and prove that 
$\sqrt{1+|\nabla_x \eta|^2}\partial_n \phi\arrowvert_{\partial\Omega}$ 
is well-defined and belongs to $H^{-\mez}(\xT^d)$. 
As a result, one can define the Dirichlet-to-Neumann operator $G(\eta)$ by
$$
G(\eta)\psi (x)=\sqrt{1+|\nabla_x \eta|^2}\partial_n\phi\arrowvert_{y=\eta(x)}
=\partial_y\phi(x,\eta(x))-\nabla_x \eta(x)\cdot\nabla_x\phi(x,\eta(x)).
$$
Let us recall two results. Firstly, 
it follows from classical elliptic regularity results that, for any 
$s\ge 1/2$, $G(\eta)$ is bounded from $H^s(\xT^d)$ into 
$H^{s-1}(\xT^d)$. This property still holds in 
the case where $\eta$ has limited regularity. 
For any $s>d/2+1$ we have
\be\label{DN:Sobolev}
\lA G(\eta)\psi\rA_{H^{s-1}}\le C\big(\lA \eta\rA_{H^{s}}\big)\lA \psi\rA_{H^{s}}.
\ee
In particular, if $\eta$ and $\psi$ belong to $C^\infty(\xT^d)$, then $G(\eta)\psi\in C^\infty(\xT^d)$.

For our purpose  a fundamental fact is that $G(\eta)$ is a positive operator. Indeed, it follows from Stokes' theorem that
\be\label{positivityDN}
\int_{\xT^d} \psi G(\eta)\psi\dx=\int_{\partial\Omega}\phi \partial_n \phi\dsigma=
\iint_{\Omega}\la\nabla_{x,y}\phi\ra^2\dydx\ge 0.
\ee

In addition to the Dirichlet-to-Neumann operator, 
we introduce the functions $B,V$ defined by,
$$
B =\partial_y \phi\arrowvert_{y=\eta}, \quad
V =(\nabla_x \phi)\arrowvert_{y=\eta},
$$
where again $\phi$ is the harmonic extension of $\psi$ given by \e{defi:phi}. 

We recall the following identities. 

\begin{lemm}\label{L:31}
We have
\be\label{n897}
B =\frac{G(\eta)\psi+\nabla_x \eta\cdot\nabla_x \psi}{1+\la \nabla_x \eta\ra^2},\qquad 
V =\nabla_x \psi-B  \nabla_x \eta.
\ee
Moreover if $h=+\infty$ then
\be\label{n991}
G(\eta)B=-\cnx V.
\ee
\end{lemm}
\begin{proof}
By definition of the operator $G(\eta)$, we have
\be\label{n997}
G(\eta)\psi=\big( \partial_y \phi-\nabla_x \eta\cdot \nabla_x \phi\big)\arrowvert_{y=\eta}
=B-\nabla_x \eta\cdot 
V. 
\ee
On the other hand  it follows from the chain rule that
$$
\nabla_x \psi=\nabla_x (\phi\arrowvert_{y=\eta})=(\nabla_x \phi)\arrowvert_{y=\eta}
+(\partial_y\phi)\arrowvert_{y=\eta}
\nabla_x \eta=V+B\nabla_x \eta. 
$$
Consequently  we obtain the wanted identity for 
$V$: 
$$
V=\nabla_x \psi -B\nabla_x \eta.
$$
Now, by reporting this formula in \e{n997} we get
$$
G(\eta)\psi=(1+\la \nabla_x \eta\ra^2)
B-\nabla_x\psi \cdot\nabla_x \eta,
$$
which immediately implies the desired result for $B$. 

The identity \e{n991} is proved in~\cite{ABZ3,BLS,LannesJAMS}.
\end{proof}

Notice that $G(\eta)\psi$ depends linearly 
in $\psi$ but nonlinearly in $\eta$. 
This is one of the main difficulty in the  study of 
water-waves equations. One key tool to study its dependence in $\eta$ is given by the following result.
\begin{prop}[from Lannes~\cite{LannesJAMS,LannesLivre}]\label{P:shape}
Let $s$ be such that $s>1+d/2$. 
Let $\psi \in H^s(\xT^d)$ and $\eta_0\in H^{s} (\xT^d)$. 
Then there is a neighborhood $\mathcal{U}_{\eta_0}\subset H^{s}(\xT^d)$ of $\eta_0$ such that
the mapping
$$
\eta\in \mathcal{U}_{\eta_0} \mapsto G(\eta)\psi \in H^{s-1}(\xT^d)
$$
is differentiable. Moreover, for all $\zeta\in H^s(\xT^d)$, we have
\be\label{n:shape}
d G(\eta)\psi \cdot  \zeta \defn
\lim_{\eps\rightarrow 0} \frac{1}{\eps}\big\{ G(\eta+\eps \zeta)\psi -G(\eta)\psi\big\}
= -G(\eta)(\mathcal{B}\zeta) -\cnx (\mathcal{V}\zeta),
\ee
where
$$
\mathcal{B}=\frac{G(\eta)\psi+\nabla_x \eta\cdot\nabla_x\psi}{1+\la \nabla_x \eta\ra^2},
\quad \mathcal{V} =\nabla_x\psi-\mathcal{B}\nabla_x \eta.
$$
\end{prop}
This result is proved for smoother functions by Lannes in~\cite{LannesJAMS}. 
We refer to his monograph~\cite{LannesLivre} for the proof 
in the general case. 

\subsection{Reformulation of the water-wave problem}

The purpose of this section is to recall 
the Craig--Sulem formulation (see \cite{CrSuSu,CrSu}) on $\eta$ and $\psi$ of the water-wave system (see also \cite{Bertinoro}).    
 
\begin{prop}\label{propN}
Set
\begin{equation}\label{N=}
  N(\psi, \eta) = \mez \vert V\vert^2 - \mez B^2 + B(V \cdot\nabla_x \eta).
  \end{equation}
Then the water-wave system can be written as
\begin{equation*}
\left\{
\begin{aligned}
&\partial_t \eta = G(\eta)\psi,\\
&\partial_t \psi + g \eta   + N(\psi, \eta) = 0.
 \end{aligned}
 \right.
 \end{equation*}
\end{prop}

\begin{proof}
By definition of $G(\eta)\psi$, 
we get the first equation directly from \eqref{kinematics}.

Now, by evaluating  the equation
$$
\partial_t \phi+ \frac{1}{2} \la \nabla_{x,y} \phi \ra^2 +  P +g y =0,
$$
on the free surface we obtain
$$
(\partial_t\psi-B\partial_t\eta)+\mez |V|^2+\mez B^2+g\eta =0.
$$
Since $\partial_t \eta=G(\eta)\psi=B-V\cdot\nabla_x\eta$, it follows that
$$
\partial_t\psi+g\eta +\mez |V|^2-\mez B^2+B (V\cdot \nabla_x \eta)=0,
$$
which gives the second equation.
\end{proof}

\begin{coro}\label{P:C1:PZ}
The water-wave equations can be written under the form
\begin{equation}\label{ZCS}
\left\{
\begin{aligned}
&\frac{\partial \eta}{\partial t} - G(\eta) \psi =0,\\
&\frac{\partial \psi}{\partial t} + g \eta + \frac{1}{2} \la \nabla_x \psi \ra^2 
- \frac{1}{2} 
\frac{(\nabla_x \psi \cdot \nabla_x \eta + G(\eta) \psi)^2}{1 + |\nabla_x \eta |^2} 
 =0.
\end{aligned}
\right.
\end{equation}
 \end{coro}
\begin{proof}
Since $V = \nabla_x \psi - B \nabla_x \eta$ we see that
$$
\vert V\vert^2 = \vert \nabla_x \psi\vert^2 + B^2 \vert \nabla_x \eta\vert^2 -2B (\nabla_x \psi  \cdot \nabla_x \eta),
$$
which implies that
$$
N(\psi, \eta)  = \mez\vert \nabla_x \psi\vert^2 + \mez B^2 \vert \nabla_x \eta\vert^2 - B (\nabla_x \psi  \cdot \nabla_x \eta)-\mez B^2 + B(V\cdot \nabla_x \eta).
$$
Now, 
the expression of $V$ gives
$$
B(V\cdot \nabla_x \eta) = B (\nabla_x \psi  \cdot \nabla_x \eta) - B^2\vert \nabla_x \eta\vert^2.
$$
Therefore
$$
N(\psi, \eta) = \mez\vert \nabla_x \psi\vert^2 - \mez B^2(1+\vert \nabla_x \eta\vert^2), 
$$
and we conclude using Proposition \ref{propN} and the expression of $B$ given by\eqref{n897}.
\end{proof}
To conclude, we recall that the mean value of $\eta$ is conserved (this is equivalent to saying that the mass of the fluid is conserved).

\begin{prop}\label{eta1}
Let $(\eta,\psi)\in C^0([0,T];H^{s}(\xT^d)\times H^s(\xT^d))$ for some $T>0$ and some 
$s>2+d/2$. Then, for all $t\in [0,T]$,
$$
\int_{\xT^d} \eta(t,x)\dx 
= \int_{\xT^d} \eta_0(x)\dx.
$$
\end{prop}
\begin{proof}
Write,
$$
\fract\int_{\xT^d} \eta\dx
=\int_{\xT^d} \partial_t\eta\dx 
=\int_{\xT^d} G(\eta)\psi \dx
$$
and then observe that the later integral vanishes, as can be verified by using the divergence theorem,
$$
\int_{\xT^d} G(\eta)\psi\dx=\int_{\partial \Omega} \partial_n\phi \dsigma=
\iint_{\Omega} \Delta_{x,y}\phi\dydx=0,
$$
where we have used the fact that $\partial_n  \phi = - \partial_y \phi =0$ on the bottom.
\end{proof}
 \section{Trace inequalities}\label{Appendix:C}
Let $h\in [1,+\infty]$ and consider $\eta\in W^{1,\infty}(\xT^d)$ satisfying $\inf_{\xT^d}\eta(x)>-h/2$ when $h<+\infty$. 
Then it follows from the classical variational theory that, 
for all function $\psi$ in the Sobolev 
space $H^{\mez}(\xT^d)$, there is a 
unique variational solution 
$\phi$ to the problem
\be\label{defi:phi-again}
\begin{aligned}
&\Delta_{x,y}\phi=0 \quad\text{in }\Omega=\{(x,y)\,:\, x\in \xT^d,~-h<y<\eta(x)\},\\ 
&\phi\arrowvert_{y=\eta}=\psi,\quad \partial_y\phi\arrowvert_{y=-h}=0.
\end{aligned}
\ee

In this section, we are going to prove several trace inequalities, which allow to control the trace of $\phi$ either on the free surface (which is $\psi$) 
or at the bottom. Such trace inequalities have been systematically studied by Lannes in his book~\cite{LannesLivre}. Here we will prove estimates which are optimal 
with respect to the dependence in $\eta$ as we will explain.

\subsection{Case of infinite depth}\label{S:BMO}

In the infinite depth case (that is for $h=+\infty$), we recall a trace estimate from a joint work of the first author with 
Quoc-Hung Nguyen (it appears in the Lecture notes~\cite[Chapter $4$]{Berkeley}). 
It gives a trace estimate assuming only that the 
gradient of $\eta$ is estimated in $\BMO(\xT^d)$. 
This space, introduced by 
John and Nirenberg, is defined as follows. 
\begin{defi}\label{defi:BMO}
The space $\BMO(\xT^d)$ (for Bounded Mean Oscillations) 
consists of those functions 
$ f\in L^1 (\xT^d)$ such that $\lA f\rA_{\BMO}<+\infty$,
with
$$
\lA f\rA_{\BMO}=\sup_{Q\subset \xT^d} \frac{1}{\la Q\ra}\int_Q \la f-f_Q\ra\dx,
$$
where the supremum is taken over all cubes $Q\subset \xT^d$ and,
$$
f_Q=\frac{1}{\la Q\ra}\int_Q  f\dx.
$$
\end{defi}

\begin{prop}\label{P:BMO}
Assume that $h+\infty$. There exists $c>0$ depending only on the dimension $d$ such that for all $\eta\in W^{1,\infty}(\xT^d)$ and for all $\psi \in \dot{H}^\mez(\xT^d)$  there holds,
$$
\iint_{\Omega} \la \nabla_{x,y}\phi(x,y)\ra^2\dydx = \int_{\xT^d}\psi(x) (G(\eta)\psi)(x)\dx\ge \frac{c}{1+\lA \nabla_x \eta\rA_{\BMO}} \lA \psi\rA_{\dot{H}^\mez(\xT^d)}^2.
$$
\end{prop}

\subsection{Case of finite depth}
In the case of finite depth, we shall prove a result slightly weaker than Proposition \ref{P:BMO}.

\begin{prop}\label{P:C3}
There exists $C>0$ depending only 
on the dimension $d$ such that for all $h>0$,  
for all Lipschitz function $\eta$ satisfying 
$\inf_{\xT^d} \eta(x)> -\mez h$, and 
for all $\psi \in \dot{H}^\mez(\xT^d)$ there holds,
\be\label{to}
\int_{\xT^d} \psi(x) (G(\eta)\psi)(x)\dx \geq \frac{C  \tanh (h)}{1+ \Vert \nabla_x \eta\Vert_{L^\infty(\xT^d)}} \Vert \psi \Vert^2_{\dot{H}^\mez(\xT^d)},
\ee
where $\tanh$ denotes the hyperbolic tangent.
\end{prop}
\begin{rema}
The dependence in $1/(1+\lA\nabla \eta\rA_{L^\infty})$ 
is optimal. Indeed, denote by $\theta$ the harmonic extension of $\eta$ in the domain $\Omega$. Then it follows from \cite[Section $2.3$]{HP22note} (or by a slight modification of the proof of Proposition~\ref{P:4.1}) that
$$
0\le \iint_{\Omega} \la\nabla_{x,y}\theta\ra^2\dydx\le  \lA \eta\rA_{L^\infty}\bla \xT^d\bra.
$$
By applying this argument $\eta=\psi=\cos(kx)$, we see that an inequality of the form
$$
\int_{\xT^d} \psi(x) (G(\eta)\psi)(x)\dx \geq \frac{C  \tanh (h)}{(1+ \Vert \nabla_x \eta\Vert_{L^\infty(\xT^d)})^\alpha} \Vert \psi \Vert^2_{\dot{H}^\mez(\xT^d)},
$$
could not hold for $\alpha<1$.
\end{rema}
\begin{rema}
Assume that $h\ge 1$. 
It immediately follows that, for all function $u\in H^1(\Omega)$, the trace $f(x)=u(x,\eta(x))$ satisfies
$$
\lA f\rA_{\dot{H}^\mez}\les \big(1+\lA \nabla_x\eta\rA_{L^\infty(\xT^d)}\big)^\mez \lA \nabla_{x,y}u\rA_{L^2(\Omega)}.
$$
Indeed, if we denote by $\phi$ the harmonic extension of $f$ (solution to \e{defi:phi-again} with $\psi$ replaced by $f$), then 
$$
\int_{\xT^d} \psi(x) (G(\eta)\psi)(x)\dx =\lA \nabla_{x,y}\phi\rA_{L^2(\Omega)}^2\le \lA \nabla_{x,y}u\rA_{L^2(\Omega)}^2
$$
where the last inequality expresses the fact that the harmonic extension minimises the Dirichlet energy.
\end{rema}
\begin{proof}
We begin by 
using an elementary diffeomorphism to flatten 
the domain $\Omega = \{(x,y): x \in \xT^d, -h<y<\eta(x)\}$. Namely, 
set
$$
\rho(x,z) = \frac{1}{h}(z+ h)\eta(x) +z.
$$
Then $\partial_z \rho = \frac{1}{h} \eta(x) +1$ 
and  $\nabla_x \rho = \frac{1}{h}(z+h) \nabla_x \eta$. 
It follows that the map
$$
\widetilde{\Omega} = \xT^d \times (-h,0)\to \Omega, \quad (x,z)  \mapsto (x, \rho(x,z)),
$$ 
is a  $W^{1, \infty}$-diffeomorphism 
whose jacobian is equal to   $\partial_z \rho(x,z) \geq \mez$.
 
Let $\widetilde{\phi}$ be the image of $\phi$ by 
this diffeomorphism  that is, $$
\widetilde{\phi}(x,z) = \phi(x, \rho(x,z)).
$$
Notice that, by this diffeomorphism, 
\begin{equation}\label{partial-tilde}
\begin{aligned}
&\partial_y  \text{ transforms to  } 
\Lambda_1 = \frac{1}{\partial_z \rho}\partial_z 
\text{ and }\\
&\partial_{x_j}  \text{ transforms to }
\Lambda_{2,j} = \partial_{x_j} -( \partial_{x_j} \rho) \Lambda_1.
\end{aligned}
\end{equation}
In addition, we have  $\widetilde{\phi}\arrowvert_{z=0} = \psi$ and $\partial_z \widetilde{\phi}\arrowvert_{z=-h} =0$.

We know (see \eqref{positivityDN}) 
that
\begin{equation}\label{Q}
\begin{aligned}
\int_{\xT^d} \psi(x) (G(\eta)\psi)(x)\dx
= &\iint_{\Omega} \vert \nabla_{x,y}\phi(x,y)\vert^2\dx \dy\\
= &\iint_{\widetilde{\Omega}}\big( (\Lambda_1 \widetilde{\phi}(x,z))^2 
+ \vert\Lambda_2 \widetilde{\phi} (x,z)\vert^2 \big)
 \,\partial_z \rho(x,z)\dx \dz := Q(\widetilde{\phi}).
\end{aligned}
\end{equation}
Let $\tanh (t\vert D_x\vert)$ be the Fourier multiplier  with symbol $\tanh(t\vert \xi\vert)$. 
Since $\tanh(0) = 0$ we can write,
\begin{equation}\label{I=}
\begin{aligned}
 I&:= \int_{\xT^d}\int_{-h}^0 \partial_z \big( \widetilde{\phi}(x,z)  \vert D_x \vert \tanh ((z+h)\vert D_x\vert)\widetilde{\phi}(x,z)\big)\dx\dz \\
 &= \int_{\xT^d} \psi(x) \vert D_x \vert \tanh(h\vert D_x\vert)\psi(x)\dx. 
 \end{aligned}
 \end{equation}  
On the other hand, since the 
operator $ \vert D_x \vert \tanh ((z+h)\vert D_x\vert)$ is 
self-adjoint we have
\begin{align*}
I&= I_1 + I_2,\\
 I_1 &= 2
 \iint_{\widetilde{\Omega}}\big[(\partial_z  \widetilde{\phi})\vert D_x \vert \tanh ((z+h)\vert D_x\vert) \widetilde{\phi}\big](x,z)\dx\dz,\\ 
I_2&=   \iint_{\widetilde{\Omega}} \big[ \widetilde{\phi}  \frac{\vert D_x\vert^2}{\big(\cosh((z+h)\vert D_x\vert))^2}\widetilde{\phi})\big](x,z)\dx\dz.
\end{align*}  
  Since $\vert D_x\vert^2 = -\Delta_x$ we have
$$
I_2
=\iint_{\widetilde{\Omega}}\Big[\frac{1}{\big(\cosh ((z+h)\vert D_x\vert))^2}\nabla_x \widetilde{\phi}\cdot \nabla_x \widetilde{\phi}\Big](x,z)\dx \dz.
$$
and so, since $\cosh(x) \geq 1$,
$$
\vert I_2 \vert \leq  \iint_{\widetilde{\Omega}}\vert \nabla_x \widetilde{\phi}(x,z)\vert^2\dx\dz.
$$
Using \e{partial-tilde}, \eqref{Q} and the fact that $\partial_z \rho \geq \mez$, we deduce that  
   \begin{equation}\label{est_I2}
   \vert I_2 \vert   \leq C  (1+ \Vert \nabla_x \eta\Vert_{L^\infty(\xT^d)})Q(\widetilde{\phi}).
   \end{equation}
   
It remains to estimate the term $I_1$. To do so, given $1\le j\le d$, 
denote by $R_j$ the 
Riesz multiplier defined 
by 
$R_j u = \mathcal{F}^{-1}\big(  \frac{-i\xi_j}{\vert \xi \vert} \widehat{u}\big)$ 
and then set $R = (R_1, \ldots, R_d)$. 
Then $\vert D_x \vert = R \cdot \nabla_x$ which allows us to write
\begin{align*}
 I_1 &=    2 \iint_{\widetilde{\Omega}} \big[(\partial_z  \widetilde{\phi})  \tanh ((z+h)\vert D_x\vert) R \cdot \nabla_x\widetilde{\phi}\big](x,z)\dx\dz\\
 &=    2 \iint_{\widetilde{\Omega}} \big[(\Lambda_1 \widetilde{\phi})  \tanh ((z+h)\vert D_x\vert) R \cdot \nabla_x\widetilde{\phi}\big](x,z)\,  \partial_z \rho(x,z) \dx\dz.
 \end{align*}
We now use again the fact that 
$\nabla_x  = \Lambda_2 + (\nabla_x \rho) \Lambda_1$ to infer that
\begin{align*}
I_1 &= I_{11} + I_{12}, \\
I_{11} &=  2 \iint_{\widetilde{\Omega}} \big[(\Lambda_1  \widetilde{\phi})  \tanh ((z+h)\vert D_x\vert) R \cdot \Lambda_2\widetilde{\phi}\big](x,z)\,   \partial_z \rho(x,z) \dx\dz,\\
I_{12}&= 2 \iint_{\widetilde{\Omega}} \big[(\Lambda_1   \widetilde{\phi})  \tanh((z+h)\vert D_x\vert) R \cdot (\nabla_x \rho)\Lambda_1\widetilde{\phi}\big](x,z)\,   \partial_z \rho(x,z) \dx\dz.
\end{align*}
By using the two elementary inequalities 
$$
0\leq \tanh(x)\leq 1 \quad\text{for}\quad 
x\geq 0,\qquad 
\vert R_j(\xi)\vert \leq 1,
$$
together with the 
Cauchy-Schwarz inequality, 
we easily obtain that
$$
\vert I_{11} \vert \leq C  Q(\widetilde{\phi}).
$$
In the same way  we get,
$$
\vert I_{12} \vert  \leq C\Vert \nabla_x \eta\Vert_{L^\infty(\xT^d)}
Q(\widetilde{\phi}).
$$
Summing up  we have
\begin{equation}\label{est_I1}
\vert I_1 \vert   \leq C  (1+ \Vert \nabla_x \eta\Vert_{L^\infty(\xT^d)})Q(\widetilde{\phi}).
\end{equation} 
Now, by combining \eqref{I=}, \eqref{est_I2}, \eqref{est_I1} together with 
\eqref{Q}  we conclude that,
$$
\int_{\xT^d} \psi(x) \vert D_x \vert \tanh(h\vert D_x\vert)\psi(x)\dx
\leq C  (1+ \Vert \nabla_x \eta\Vert_{L^\infty(\xT^d)})
\int_{\xT^d} \psi(x) (G(\eta)\psi)(x)\dx.
$$
On the other hand, the Plancherel theorem implies that
$$
\int_{\xT^d} \psi(x) \vert D_x \vert \tanh (h\vert D_x\vert)\psi(x)\dx 
= c_d \sum_{\xi \in \xZ^d}\vert \xi \vert \tanh(h \vert \xi \vert) \vert \widehat{\psi}(\xi)\vert^2.
$$
Since $\tanh(h \vert \xi \vert) \geq \tanh(h)$ for $\xi\in \xZ^d\setminus\{0\}$, 
we conclude that
$$
\int_{\xT^d} \psi(x) (G(\eta)\psi)(x)\dx \geq 
\frac{C\tanh(h)}{1+ \Vert \nabla_x \eta\Vert_{L^\infty(\xT^d)}} \sum_{\xi \in \xZ^d}\vert \xi \vert  \vert \widehat{\psi}(\xi)\vert^2,
$$
equivalent to the wanted result.
\end{proof}

\subsection{A trace inequality at the bottom}
\begin{prop}\label{fond-inf1} 
Let $h\in [2,+\infty)$ and  consider two functions $\psi \in H^\mez(\xT^d)$, $\eta \in W^{1,\infty}(\xT^d)$ with $\int_{\xT^d}\eta(x)\dx=0$ 
and  $\inf_{x\in \xT^d}\eta(x)> - \frac{h}{3}$. Let 
$\phi_h$ be the unique variational solution to the  problem,
\begin{align*}
 &\Delta_{x,y} \phi_h = 0\text{ in } \Omega= \{(x,y): x \in \xT^d, -h<y<\eta(x)\},\\
 &\phi_h\arrowvert_{y = \eta(x)} = \psi(x), \quad\partial_y\phi_h \arrowvert_{y= -h} =0.
\end{align*}
There  exists a constant $C>0$  
depending only on $d$ such that, for all  $h \geq 2$,
$$
 \int_{\xT^d} \vert \nabla_x \phi_h(x, -h)\vert^2\dx \leq 
C\big(1+\lA \nabla_x\eta\rA_{L^\infty}^3\big) 
e^{-\frac{1}{4} h} \Vert \psi\Vert^2_{H^\mez(\xT^d)}.
$$
\end{prop}
\begin{proof}
 
{\em Step} $0$.--- We begin by defining a lifting $\underline {\psi} \in H^1(\Omega)$ 
of the boundary data $\psi$, such that
\begin{equation}\label{fd-inf1}
\begin{aligned}
&\underline {\psi}\arrowvert_{y=\eta(x)}
= \psi(x), \quad \underline {\psi}(x,y)= 0 
\text{ for }y\leq -\frac{h}{2} \quad \text{ and},\\
&\Vert \underline {\psi}\Vert_{H^1(\Omega)}
\leq C(1+ \Vert \nabla_x \eta\Vert_{L^\infty(\xT^d)})  \Vert \psi \Vert_{H^\mez(\xT^d)}.
\end{aligned}
\end{equation}
To do so, introduce a cut-off function 
$\chi \in C^\infty(\xR)$ such that
$$
\chi(z) = 0 \text{ if } z \leq -\frac{h}{6},\quad 
\chi(z) = 1 \text { if } z \geq -\frac{h}{8}, \quad 
0 \leq \chi \leq 1,
$$
and then set, for all $x\in \xT^d$ and all $z\leq 0$,
$$
\underline{\widetilde{\psi}}(x,z) = \chi(z)e^{z\vert D_x \vert} \psi.
$$
Since $ \int_{-\infty}^0 e^{z \vert \xi \vert}\dz= \frac{1}{\vert \xi \vert}$, 
one easily sees that there exists a constant $C>0$, 
depending only on the dimension $d$  such that,
$$
\Vert \underline{\widetilde{\psi}}\Vert_{H^1(\xT^d \times (-\infty,0))} \leq C \Vert \psi \Vert_{H^\mez(\xT^d)}.
$$
Now, given $(x,y)\in \Omega$, set
$$
\underline {\psi}(x,y) =  \underline{\widetilde{\psi}}(x, y-\eta(x)).
$$
It is easy to check that $\underline {\psi}$ satisfies all the requirements in \eqref{fd-inf1}.

{\em Step} $1$.--- Let  $f_h(x) = \phi_h(x,-h/2)$. 
We claim that there  exists  $C>0$
depending only on $d$ such that, for all  $h \geq 2$,
\begin{equation}\label{fd-inf0}
\lA f_h\rA_{L^2(\xT^d)}^2\le 
C(h^2 + \lA \nabla_x\eta\rA_{L^\infty(\xT^d)}^3)
\Vert \psi \Vert_{H^\mez(\xT^d)}^2.
\end{equation}
   
To prove this claim, we decompose $\phi_h$ as 
$\phi_h = u_h + \underline{\psi}$. Also, $u_h$ solves
\be\label{eq:uh}
\left\{
\begin{aligned}
&\Delta_{x,y} u_h = - \Delta_{x,y}\underline{\psi} \quad\text{ in } \Omega= \{-h<y<\eta(x)\},\\ &u_h\arrowvert_{y = \eta(x)} = 0, \quad\partial_yu_h \arrowvert_{y= -h} =0.
\end{aligned}
\right.
\ee
Notice that, 
since $\underline{\psi}(x,-h/2)=0$, we have $f_h = u_h \arrowvert_{y=-h/2}$. Also, 
since $u_h$ vanishes on $\{y=\eta(x)\}$ 
we have 
$$
\int_{\xT^d}f_h(x)^2\dx =-\int_{\xT^d}\int_{-h/2}^{\eta(x)}\partial_y(u_h(x,y)^2)\dy\dx.
$$
Thus, it follows from the Cauchy-Schwarz inequality that
\be\label{est:fhn1}
\lA f_h\rA_{L^2(\xT^d)}^2\le 2\lA u_h\rA_{L^2(\Omega)}\lA \partial_yu_h\rA_{L^2(\Omega)}.
\ee

Now, since $\Omega$ is contained in a strip, we can apply the Poincar\'e inequality, which 
ensures that there exists $C>0$ depending only on the dimension $d$ such that,
\begin{equation}\label{poinc1}
\lA u_h \rA_{L^2(\Omega)}\leq C(h + \lA \eta \rA_{L^\infty})
\lA\partial_y u_h\rA_{L^2(\Omega)}.
\end{equation}
On the other hand, by the variational theory applied to \e{eq:uh}, using the estimates 
\eqref{poinc1} and \eqref{fd-inf1}, we obtain 
\begin{equation}\label{fd-inf2}
\begin{aligned}
  \Vert \nabla_{x,y} u_h \Vert_{L^2(\Omega)} &\leq C(h+ \Vert \eta\Vert_{L^\infty(\xT^d)}) \Vert \underline{\psi}\Vert_{H^1(\Omega)}\\
  &\leq C' (h+ \Vert \eta\Vert_{L^\infty(\xT^d)})(1+\Vert \nabla_x\eta\Vert_{L^\infty(\xT^d)}) \Vert \psi \Vert_{H^\mez(\xT^d)},
  \end{aligned}
\end{equation}
where  $C'$ is independent of $h$.

Now, recall that, by hypothesis, we have $\int_{\xT^d} \eta(x)\dx = 0$. 
It follows that there exists $x_0 \in \xT^d$ such that $\eta(x_0) = 0$. 
So, writing,
$$
\eta(x) - \eta(x_0)
= \int_0^1 \nabla_x \eta(\lambda x+ (1-\lambda)x_0)(x-x_0)\, d\lambda,
$$
we immediately see that there exists $C>0$ depending only on $d$ such that,
\begin{equation}\label{poinc2}
\Vert \eta\Vert_{L^\infty(\xT^d)}\leq C\Vert \nabla_x\eta\Vert_{L^\infty(\xT^d)}.
\end{equation}

By combining the previous estimates, we obtain the claim~\e{fd-inf0}.

{\em Step} $2$ --- Denote by 
$\theta_h$ the unique solution to the problem
$$
\left\{
\begin{aligned}
&\Delta_{xy} \theta_h = 0 \text{ in }  S_h: =\{(x,y): x \in \xT^d, -h<y<-h/2\}, \\ &\theta\arrowvert_{y=-m} = f_h, \quad \partial_y \theta_h \arrowvert_{y = -h} = 0.
\end{aligned}
\right.
$$
Since $\phi_h$ and $\theta$ are two harmonic functions which coincide on $\{y=-h/2\}$, we have 
\begin{equation}\label{fd-inf4}
\phi_h = \theta_h, \quad \text{ in } S_h.
\end{equation}
Performing a Fourier transform  with respect to  $x$ we see easily that,  
$$ \widehat{\theta_h}(\xi,y) = \frac{ \cosh((y+h)\vert \xi \vert)}{\cosh((h/2)\vert \xi \vert) } \widehat{f_h}(\xi),$$ 
so that,
\begin{equation}\label{fd-inf5}
  \vert \xi \vert\widehat{\theta_h}(\xi,-h) = \frac{   \vert \xi \vert\widehat{f_h}(\xi)}{\cosh((h/2)\vert \xi \vert) }. 
  \end{equation}

Using Parseval's identity, \eqref{fd-inf4}, \eqref{fd-inf5} 
and the fact that $\cosh((h/2)\vert \xi \vert) \geq \mez e^{(h/2)\vert \xi \vert}$  we can write,
\begin{equation}\label{fd-inf6}
\begin{aligned}
 h^2\int_{\xT^d} \vert \nabla_x \phi_h(x, -h)\vert^2\dx &\leq  C_d \sum_{\xi \in \xZ^d} h^2 \vert \xi \vert^2 e^{-\frac{h}{2}\vert \xi \vert} \vert \widehat{f_h}(\xi)\vert^2\\
 & \leq C_d \sum_{\xi \in \xZ^d\setminus\{0\}} h^2 \vert \xi \vert^2 e^{-\mez h\vert \xi \vert} \vert \widehat{f_h}(\xi)\vert^2,
 \end{aligned}
 \end{equation}
since $h \geq 2$. Now, directly from the definition of the Fourier transform, we have
$$
\vert \widehat{f_h}(\xi)\vert \leq \Vert  \widehat{f_h}\Vert_{L^\infty(\xZ^d)} \leq \Vert f_h\Vert_{L^1(\xT^d)} \leq  C(d)\Vert f_h\Vert_{L^2(\xT^d)}.
$$
Since for $\xi \in \xZ^d\setminus\{0\}$ we have $e^{-\frac{1}{4}h\vert \xi \vert} \leq e^{-\frac{1}{4}h   }$ and $\sum_{\xi \in \xZ^d\setminus\{0\}} h^2 \vert \xi \vert^2 e^{-\frac{1}{4} h\vert \xi \vert}\le C$, for $h \geq 2$,
we deduce the desired result from \eqref{fd-inf0}.
\end{proof}

\section{Comparison with other virial type  theorems}\label{S:Virial}

\subsection{The first virial theorem of Clausius}
The first virial theorem was proved by Clausius in his famous paper of 1870 (\cite{Clausius}). He states his main result in the following form: {\em The mean vis viva of the system is equal to its virial}. Here  {\em vis viva} is an old Latin expression  which means 'living force', 
invented by Leibniz to designate kinetic energy 
while {\em virial} is a word invented by Clausius to designate a modified version of potential energy. To explain this statement  let us consider the following Hamiltonian system, 
\be\label{ODE}
\left\{
\begin{aligned}
&\fract q(t)=p(t) \quad (q(t)\in \xR),\\
&\fract p(t)=-V'(q(t))\quad\text{with}\quad V(q)=\frac{1}{2m}q^{2m} \quad(m\in \xN^*).
\end{aligned}
\right.
\ee
With $I(t)\defn\mez p(t)q(t)$ we verify that,
\be\label{virial0}
\fract I(t)=\mez p(t)^2- \mez q(t)V'(q(t)).
\ee
Since $t\mapsto I(t)$ is bounded globally in time, this implies that, for all time $T>0$, 
$$
\la\frac{1}{T}\int_0^T\left( \mez p(t)^2-q(t)V'(q(t))\right)\dt\ra=\la \frac{I(T)-I(0)}{T}\ra\le \frac{C}{T}.
$$
Hence, as $T$ 
goes to $+\infty$, the mean value of the kinetic energy $\mez p(t)^2$ is equal to the mean value of the virial, which is by definition the quantity $\mez q(t)V'(q(t))$. 
Notice that, when $m=1$, the equation~\e{ODE} is linear and the virial coincide with the potential energy $\mez q(t)^2$. In this case the virial theorem states that the mean 
kinetic energy is equal to the mean potential energy.

In this article, we prove exact identities of the form~\e{virial0} for solutions of the incompressible Euler equations with free surface. Notice that our main identity (see~\e{MI1}) goes in a reverse direction with respect to the Clausius theorem: we prove that the mean value of the potential energy is equal to a modified version of the kinetic energy.


\subsection{The virial theorem of Benjamin and Olver}

Virial theorems are related to the classical identities obtained by the multiplier method. The relation 
with the identities of Pohozaev and Derrick 
are evoked by Berestycki and Lions in their famous article~\cite{Berestycki-Lions}. In this paragraph, we will 
use this point of view to generalize an identity obtained by Benjamin and Olver in their thorough study of conservation laws fo water waves (\cite{BO}). To do this, guided by \cite{AIT}, we will combine Theorem~\ref{T:virial} with another identity obtained related to the conservation of momentum. 

In this paragraph, we consider functions 
defined on $\xR$ (that is the spatial 
dimension is $d=1$)
and partial differentiations are denoted by suffixes so that $f_t=\partial_t f$, $f_x=\partial_xf$ and $f_y=\partial_y f$.
Also, in this 
appendix we perform only formal computations and consider functions 
which are decaying at infinity instead 
of being periodic in $x$.

Let us recall that Noether's theorem and the invariance with respect to horizontal translations imply that the horizontal momentum is conserved. Namely, we have
\begin{equation*}
\fract \int_{\xR} \eta \psi_x \dx=0.
\end{equation*}
As explained in \cite{AIT}, there are several different local conservation laws associated with the conservation of momentum. In particular, there holds
\be\label{AIT1}
\partial_t(\eta\psi_x)+\partial_x S=0,
\ee
with
$$
S(t,x)\defn -\eta(t,x)\psi_t(t,x)-\frac{g}{2}\eta^2(t,x)+
\mez \int_{-h}^{\eta(t,x)}(\phi_x^2-\phi_y^2)(t,x,y)\dy.
$$

We will use \e{AIT1} 
to prove that
\be\label{claim:pohozaev}
\fract \int_{\xR} x\eta \psi_x \dx=\iint (\phi_x^2-\phi_y^2)\dydx+\frac{g}{2}\int \eta^2\dx-\frac{h}{2}\int_\xR\phi_x(t,x,-h)^2\dx.
\ee
To do so, multiply \e{AIT1} by $x$ and integrate by parts to obtain
$$
\fract \int_{\xR} x\eta \psi_x \dx=\int_{\xR}S\dx.
$$
Now, it follows from the equation for $\psi$ that
$$
-\eta\psi_t -\frac{g}{2}\eta^2=\frac{g}{2}\eta^2+\eta N.
$$
Therefore, using the identity~\e{mez-3mez2}, we get
$$
\int_{\xR}S\dx=\frac{g}{2}\int_\xR\eta^2\dx+\iint (\phi_x^2-\phi_y^2)\dydx -\frac{h}{2}\int_\xR\phi_x(t,x,-h)^2\dx.
$$

Now recall that (see~\e{MI1}),
\begin{equation}\label{MI1R}
\begin{aligned}
\fract
\int_{\xR} \eta\psi\dx &= \iint_{\Omega(t)} 
\Big(\frac{3}{2} \phi_y^2  + \frac{1}{2} 
\phi_x^2\Big)\dydx\\
&\quad- g\int_{\xR}\eta^2\dx+  \frac{h}{2}\int_{\xR} \vert \nabla_x \phi(t, x, -h)\vert^2\dx. 
\end{aligned}
\end{equation}
By combining \e{claim:pohozaev} and \e{MI1R}, we find that, for any $\lambda\in\xR$, 
\be\label{claim:pohozaev2}
\begin{aligned}
\fract \int_{\xR} (x\eta \psi_x+\lambda \eta\psi) \dx
&=\iint \Big(1+\frac{\lambda}{2}\Big)\phi_x^2+\Big(\frac{3\lambda}{2}-1\Big)\phi_y^2)\dydx\\
&\quad+\Big(\frac{g}{2}-\lambda g\Big)\int \eta^2\dx\\
&\quad+\Big(\frac{\lambda h}{2}-\frac{h}{2}\Big)\int_\xR\phi_x(t,x,-h)^2\dx.
\end{aligned}
\ee
When $\lambda=2$  we have,
$$
1+\frac{\lambda}{2}=\frac{3\lambda}{2}-1=2,
$$
and the previous identity reads,
\be\label{claim:pohozaev3}
\fract \int_{\xR} (x\eta \psi_x+2 \eta\psi) \dx
=4E_k -3E_p+\frac{h}{2}\int_\xR\phi_x(t,x,-h)^2\dx.
\ee
Moreover, integrating by part,  we see that
$$
\int_{\xR} (x\eta \psi_x+2 \eta\psi) \dx=\int_{\xR}(\eta-x\eta_x)\psi\dx.
$$
We thus conclude that
$$
\fract \int_{\xR}(\eta-x\eta_x)\psi\dx
=4E_k -3E_p+\frac{h}{2}\int_\xR\phi_x(t,x,-h)^2\dx.
$$
Therefore, we recover in the special case $\lambda=2$ an identity termed 'virial' by Benjamin and Olver~\cite{BO} (this identity appears on top of page $167$ in \cite{BO}). 
In their paper (see \S$6.3$ and the discussion on page 139 in \cite{BO}), they observed that~\e{claim:pohozaev3} cannot be used to relate the average kinetic and potential energies, except to deduce some information for special solutions (namely solitary waves). As we explained in the introduction and in Section~\S\ref{S:3}, the solution to this problem is that the average potential energy is equal to a modified version of the kinetic energy.

\bibliographystyle{plain}

\begin{flushleft}
Thomas Alazard, Ecole Normale Sup\'erieure Paris-Saclay, CNRS
Centre Borelli UMR9010, 4 avenue des Sciences, F-91190 Gif-sur-Yvette.

\vspace{3mm}

Claude Zuily, Laboratoire de Math\'ematiques d'Orsay, Universit\'e
  Paris-Saclay, B\^atiment~307, 91405,  Orsay.

\end{flushleft}
 
\end{document}